\documentclass{amsart}

\usepackage{amsmath}
\usepackage{latexsym,amsfonts,amssymb,mathrsfs}
\usepackage{mathdots}
\usepackage{color}
\usepackage [all,2cell,color]{xy}
\usepackage{enumitem} 
\usepackage{ dsfont } 

\usepackage{bm} 

\usepackage{tikz}
\usetikzlibrary{cd}
\usetikzlibrary{calc}
\usetikzlibrary{math}
\usetikzlibrary{arrows}
\usetikzlibrary{fit}
\usetikzlibrary{positioning}
\usetikzlibrary{decorations.pathmorphing}
\usetikzlibrary{decorations.pathreplacing}
\usetikzlibrary{ decorations.markings} 
\tikzset{Rightarrow/.style={double equal sign distance,>={Implies},->},
triple/.style={-,preaction={draw,Rightarrow}},
quadruple/.style={preaction={draw,Rightarrow,shorten >=0pt},shorten >=1pt,-,double,double
distance=0.2pt}}

\pgfarrowsdeclarecombine{twotriang}{twotriang}{stealth}{stealth}%
{stealth}{stealth}
\tikzset{monohook/.style={right hook-stealth}} 
\tikzset{mono/.style={>-stealth}} 
\tikzset{epi/.style={-twotriang}} 
\tikzset{twoarrowlonger/.style={double,double distance=1.5pt,
shorten <=5pt,shorten >=6pt,
decoration={markings,mark=at position -4pt with {\arrow[scale=1.75]{>}}},
preaction={decorate}}} 
\tikzstyle{category1} = [rectangle, rounded corners, minimum width=5cm, minimum height=0.8cm,text centered, draw=black, 
text width=5cm]
\tikzstyle{category3} = [rectangle, rounded corners, minimum width=3.5cm, minimum height=0.8cm,text centered, draw=black,  
text width=3.5cm]
\tikzstyle{category4} = [rectangle, rounded corners, minimum width=3.5cm, minimum height=0.8cm,text centered, draw=black, 
text width=3.5cm]
\tikzstyle{category5} = [rectangle, rounded corners, minimum width=5cm, minimum height=0.8cm,text centered, draw=black, 
text width=5cm]
\tikzstyle{category7} = [rectangle, rounded corners, minimum width=6cm, minimum height=0.8cm,text centered, draw=black, 
text width=6cm]
\tikzstyle{higharrow}=[thick,red]

\usepackage{mathtools} 

\usepackage[
textwidth=3cm,
textsize=small,
colorinlistoftodos]
{todonotes}

\usepackage[twoside=false,left=4cm,right=4cm,top=3cm,bottom=3cm]{geometry}

\theoremstyle{plain}   
\newtheorem{thm}{Theorem}[section] 
\makeatletter\let\c@thm\c@thm\makeatother

\makeatletter\let\c@cor\c@thm\makeatother

\makeatletter\let\c@lem\c@thm\makeatother
\newtheorem{prop}{Proposition}[section]
\makeatletter\let\c@prop\c@thm\makeatother

\makeatletter\let\c@claim\c@thm\makeatother
\newtheorem{expectation}{Expectation}[section]
\makeatletter\let\c@expectation\c@thm\makeatother

\newtheorem{question}{Question}[section]
\makeatletter\let\c@question\c@thm\makeatother

\theoremstyle{definition}

\makeatletter\let\c@defn\c@thm\makeatother
\newtheorem{const}{Construction}[section]
\makeatletter\let\c@const\c@thm\makeatother

\makeatletter\let\c@notn\c@thm\makeatother

\theoremstyle{remark}

\newtheorem{rmk}{Remark}[section]
\makeatletter\let\c@rmk\c@thm\makeatother

\makeatletter\let\c@ex\c@thm\makeatother

\makeatletter\let\c@observationn\c@thm\makeatother

\makeatletter\let\c@digression\c@thm\makeatother

\makeatletter
\let\c@equation\c@thm
\numberwithin{equation}{section}
\makeatother

\usepackage{hyperref}
\usepackage[capitalise]{cleveref}

\newcommand{\newrefformat}[2]{}

\crefname{lem}{Lemma}{Lemmas}
\crefname{thm}{Theorem}{Theorems}
\crefname{defn}{Definition}{Definitions}
\crefname{notn}{Notation}{Notations}
\crefname{const}{Construction}{Constructions}
\crefname{prop}{Proposition}{Propositions}
\crefname{rmk}{Remark}{Remarks}
\crefname{cor}{Corollary}{Corollaries}
\crefname{equation}{Display}{Displays}
\crefname{ex}{Example}{Examples}
\crefname{expectation}{Expectation}{Expectations}

\newcommand{\cC}{\mathcal{C}}

\newcommand{\cE}{\mathcal{E}}

\newcommand{\cG}{\mathcal{G}}

\newcommand{\cP}{\mathcal{P}}

\newcommand{\cS}{\mathcal{S}}

\newcommand{\cW}{\mathcal{W}}

\newcommand{\cat}{\cC\!\mathit{at}}

\newcommand{\gpd}{\cG\!\mathit{pd}}

\DeclareMathOperator{\Ob}{Ob}

\DeclareMathOperator{\Hom}{Hom}
\DeclareMathOperator{\Map}{Map}

\DeclareMathOperator{\Ar}{Ar}
\DeclareMathOperator{\Seq}{Seq}

\DeclareFontFamily{OT1}{pzc}{}
\DeclareFontShape{OT1}{pzc}{m}{it}{<-> s * [1.10] pzcmi7t}{}
\DeclareMathAlphabet{\mathpzc}{OT1}{pzc}{m}{it}

\DeclareMathOperator{\op}{op}

\tikzset{arrow/.style={-stealth}} 
\tikzset{arrowshorter/.style={-stealth, shorten <=2pt, shorten >=2pt}}

\title{Waldhausen's $S_\bullet$-construction}

\author{Viktoriya Ozornova}
\address{Max Planck Institute for Mathematics, Bonn, Germany}
\email{viktoriya.ozornova@mpim-bonn.mpg.de} 

\author{Martina Rovelli}
\address{Department of Mathematics and Statistics, University of Massachusetts Amherst, Amherst, USA}
\email{mrovelli@umass.edu} 

\begin{document}

\maketitle

\begin{abstract}
This note is an expository contribution for a proceedings volume of the workshop \emph{Higher Segal Spaces and their Applications to Algebraic $K$-theory, Hall Algebras, and Combinatorics}.
We survey various versions of Waldhausen's $S_\bullet$-construction and the role they play in defining $K$-theory, and we discuss their $2$-Segality properties.
\end{abstract}

\section*{Introduction}

A prominent invariant for a ring (spectrum) $R$ is its $K$-theory space -- which is in fact a spectrum -- $K(R)$.
Various constructions are available to access $K(R)$, most notably:
Quillen's \emph{$+$-construction} \cite{QuillenCohomologyK}, Quillen's \emph{$Q$-construction}  \cite{QuillenK}, Segal's \emph{group completion} \cite{SegalCohomology}, Grayson's \emph{$S^{-1}S$-construction} \cite{Grayson}, and Waldhausen's \emph{$S_\bullet$-construction} \cite{waldhausen}.
The latter allows for some of the most general kinds of input, and it is the focus of this note.

From any ring $R$, one can consider the exact category $\cE_R$ of finitely generated projective $R$-modules. Further, given $\cE$ any exact ($\infty$-)category -- or even something more general such as a Waldhausen ($\infty$-)category -- its $K$-theory space $K(\cE)$ can be obtained through the realization of a certain simplicial space: Waldhausen's \emph{$S_\bullet$-construction} $S_\bullet(\cE)$. 

A lot of the \emph{structure} of $\cE$ -- such as information about the bicartesian squares in $\cE$ -- is naturally retained in the simplicial space $S_\bullet(\cE)$. Similarly, a lot of the relevant \emph{properties} of $\cE$ -- such as the fact that bicartesian squares are determined by the span of which they are a colimit cone -- are manifested in an interesting way in $S_\bullet(\cE)$: they are encoded in the fact that $S_\bullet(\cE)$ is a (lower) \emph{$2$-Segal space}.
Roughly, $S_\bullet(\cE)$ is a $2$-Segal space if $\cE$ is \emph{pointed} and \emph{stable} in an appropriate sense, and it is a lower $2$-Segal space if $\cE$ is pointed and satisfies just one half of the stability requirements, which we may refer to as \emph{semi-stable}.

The notion of a $2$-Segal space was discovered independently in \cite{DKbook} (cf.~the contribution \cite{SternChapter}) and in \cite{GCKT1} under the name of \emph{decomposition space} (cf.~the contribution \cite{HackneyChapter}), and the $S_\bullet$-construction $S_\bullet\cE$ was identified as the main example by both sets of authors. It was later shown in \cite{BOORS1,BOORS3} (cf.~the contribution \cite{RovelliChapter}) that in fact the $S_\bullet$-construction defines an equivalence between $2$-Segal spaces and certain double Segal spaces. Hence, every $2$-Segal space does indeed arise as an appropriate $S_\bullet$-construction.

In most cases of interest for $\cE$, the $K$-theory space $K(\cE)$ can be delooped infinitely many times, so it is in fact an $\Omega$-spectrum. This can be done by realizing an appropriate iterated version of the $S_\bullet$-construction, which produces a multisimplicial space $S^{(n)}_{\bullet,\dots,\bullet}(\cE)$.

Again, a lot of the structure of $\cE$ is manifested as a $2$-Segal property in each simplicial variable individually, so that it should be a multisimplicial space that is (lower) $2$-Segal in each simplicial variable. Enhancing the previous picture, $S^{(n)}_{\bullet,\dots,\bullet}(\cE)$ should be an $n$-fold $2$-Segal space if $\cE$ satisfies properties of pointedness and stability, and it should be an $n$-fold lower $2$-Segal space if $\cE$ satisfies only one half of the stability requirements.

\begin{center}
    \begin{tabular}{c|cc}
      $\cE$& $S_\bullet(\cE)$  &  $S^{(n)}_{\bullet,\dots,\bullet}(\cE)$\\
      \hline
 stable pointed    &   $2$-Segal &$n$-fold $2$-Segal\\
semi-stable pointed     & lower $2$-Segal&$n$-fold lower $2$-Segal
    \end{tabular}
\end{center}

In this note we discuss the various flavors of $S_\bullet$-constructions, how they can be used to define $K$-theory, their $2$-Segality properties, and how this is remembering features of the original input.

\subsection*{Acknowledgements}

We are grateful to Julie Bergner, Joachim Kock, Mark Penney, and Maru Sarazola for organizing a wondeful workshop at the Banff International Research Station: \emph{Higher Segal Spaces and their Applications to Algebraic $K$-theory, Hall Algebras, and Combinatorics} (24w5266, originally 20w5173). We are also thankful to Claudia Scheimbauer for creating and letting us use the picture macros from our previous papers and to Lennart Meier for helpful discussions on the history of the subject.
MR acknowledges the support from the National Science Foundation under Grant No. DMS-2203915.

\section{Waldhausen's $S_\bullet$-construction and $K$-theory space}

We review Waldhausen's $S_\bullet$-construction and some of its generalizations.

\subsection{Waldhausen's $S_\bullet$-construction}

We discuss in detail what Wald\-hau\-sen's $S_\bullet$-construction looks like and how it is used to define $K$-theory in the main example: an exact category of modules over a ring.

\subsubsection{The sequence construction}

Let $R$ be a nice (unital associative commutative) ring and let $\cE_R$ be the category of modules over $R$
or a full subcategory thereof.

The easiest case to keep in mind throughout this section is the case of $R$ being any ring and $\cE_R$ being the category of all $R$-modules, which is a primary example of an \emph{abelian} category. However, this choice leads to a trivial $K$-theory (cf.~\cref{KtheoryRing}). Alternatively, one can consider the case of $R$ being a local regular ring and $\cE_R$ being the category of finitely generated $R$-modules, which is also an abelian category.
Overall, what is considered the most interesting case is that of $R$ being any ring and $\cE_R$ being the category of finitely generated projective $R$-modules, which fails to be an abelian category, but can be given the structure of an \emph{exact} category.

When drawing diagrams valued in $\cE_R$, we follow the convention that arrows drawn horizontally, resp.~vertically, as
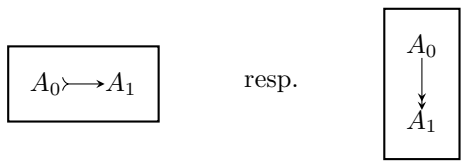
\begin{center}
\begin{tikzpicture}[inner sep=0cm, baseline=-30pt]
\begin{scope}
  \draw[thick] (-0.5,0.5) rectangle (1.5,-0.5);
\draw (0,0) node (a00){$A_0$};
\draw (1,0) node (a01){$A_1$};
        \draw[mono] (a00)--(a01);
        \end{scope}
\draw (3, 0) node (text){resp.}; 
\begin{scope}[xshift=5cm, yshift=0.5cm]
  \draw[thick] (-0.5,0.5) rectangle (0.5,-1.5);
\draw (0,0) node (a00){$A_0$};
\draw (0,-1) node (a01){$A_1$};
        \draw[epi] (a00)--(a01);
        \end{scope}
\end{tikzpicture}
\end{center}
correspond to an injective (resp.~surjective) morphism in $\cE_R$, that displayed squares of the form
    
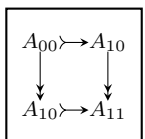
\begin{center}
\begin{tikzpicture}[scale=0.9, inner sep=0pt, font=\footnotesize, baseline=-30pt]
\begin{scope}[xshift=0cm]
\draw[thick] (2.5,0.5) rectangle (4.5,-1.5);
\begin{scope}
\draw  (3, -1)   node(a12){{\footnotesize $A_{10}$}};
\draw  (3, 0)   node(a02){{\footnotesize $A_{00}$}};
\draw  (4, 0)   node (a03){{\footnotesize $A_{10}$}};
\draw (4, -1)   node (a13){{\footnotesize $A_{11}$}};
\draw[mono] (a12)--(a13);
\draw[mono] (a02)--(a03);
\draw[epi] (a02)--(a12);
\draw[epi] (a03)--(a13);
\end{scope}
\end{scope}
\end{tikzpicture}
\end{center}
are bicartesian squares, and that objects denoted $Z_i$ are $R$-trivial modules, i.e., modules with one element.

Some first relevant information about the ring $R$ which is naturally stored in $\cE_R$ is the understanding of sequences of subobjects.
For instance, if $\cE_{\mathbb Z}$ is the category of finitely generated projective abelian groups, there are objects which admit non-trivial chains of subobjects of arbitrary length. Instead, if $\cE_{\mathbb Z/2}$ is the category of finitely generated $\mathbb Z/2$-modules, the length of a chain of subobjects of a given object is bounded by the rank of this object.
So one could consider:

\begin{const}
\label{discSeq}
    For $n\geq0$, let $\Seq^{\mathrm{disc}}_k(\cE_R)$ denote the set
    of all sequences of subobjects in $\cE_R$ of length $k$. For instance, for $k=3$ the generic element is of the form
    \begin{center}
\begin{tikzpicture}
        \draw[thick] (0,0.5) rectangle (4,-0.5);
     \begin{scope}[xshift=0cm, yshift=0cm]
        \draw (1,0) node(a00){$A_1$};
        \draw (2,0)  node(a01){$A_2$};
        \draw (3, 0)  node(a02){$A_3$};
        \draw[mono] (a00)--(a01);
        \draw[mono] (a01)--(a02);
    \end{scope}
    \end{tikzpicture}    
    \end{center}
    For $0<i\leq k$, let the $i$-th \emph{face map} $d_i\colon \Seq^{\mathrm{disc}}_k\cE_R\to \Seq^{\mathrm{disc}}_{k-1}(\cE_R)$ be given by skipping the $i$-th term. For instance, if $k=3$ this gives:
        \begin{center}
\begin{tikzpicture}
        \draw[thick] (0.5,0.5) rectangle (3.5,-0.5);
     \begin{scope}[xshift=0cm, yshift=0cm]
        \draw (1,0) node(a00){$A_1$};
        \draw (2,0)  node(a01){$A_2$};
        \draw (3, 0)  node(a02){$A_3$};
        \draw[mono] (a00)--(a01);
        \draw[mono] (a01)--(a02);
    \end{scope}
     \draw (3.8, 0) node (eq){$\xmapsto{d_i}$}; 
     \draw (6.5, 0) node (br){$
     \left\{
     \phantom{\begin{array}{cc}
A_2 \subseteq A_3, & i=1,\\
A_1 \subseteq A_3, & i=2,\\
A_1 \subseteq A_2, & i=3,\\
A\\
A\\
A\\
A_2/A_1 \subseteq A_3/A_1, & i=0
\end{array}
}
\right.
$};
\begin{scope}[yshift=0.1cm]
    \draw[thick] (4.5,1.5) rectangle (6.5,0.5);
     \begin{scope}[xshift=3cm, yshift=1cm]
        \draw (2,0)  node(a01){$A_2$};
        \draw (3, 0)  node(a02){$A_3$};
        \draw[mono] (a01)--(a02);
    \end{scope}
\draw (7.3, 1) node (i1){$,\quad i=1$};   
\end{scope}
\begin{scope}[yshift=-1.0cm]
    \draw[thick] (4.5,1.5) rectangle (6.5,0.5);
     \begin{scope}[xshift=3cm, yshift=1cm]
        \draw (2,0)  node(a01){$A_1$};
        \draw (3, 0)  node(a02){$A_3$};
        \draw[mono] (a01)--(a02);
    \end{scope}
\draw (7.3, 1) node (i2){$,\quad i=2$};   
\end{scope}
\begin{scope}[yshift=-2.1cm]
    \draw[thick] (4.5,1.5) rectangle (6.5,0.5);
     \begin{scope}[xshift=3cm, yshift=1cm]
        \draw (2,0)  node(a01){$A_1$};
        \draw (3, 0)  node(a02){$A_2$};
        \draw[mono] (a01)--(a02);
    \end{scope}
\draw (7.3, 1) node (i3){$,\quad i=3$};   
\end{scope}
    \end{tikzpicture}    
    \end{center}
Let the $0$-th
\emph{face map} $d_0\colon \Seq^{\mathrm{disc}}_k(\cE_R)\to \Seq^{\mathrm{disc}}_{k-1}(\cE_R)$ be given by taking quotients modulo the first term. For instance, if $k=3$ this gives:
        \begin{center}
\begin{tikzpicture}
        \draw[thick] (0.5,0.5) rectangle (3.5,-0.5);
     \begin{scope}[xshift=0cm, yshift=0cm]
        \draw (1,0) node(a00){$A_1$};
        \draw (2,0)  node(a01){$A_2$};
        \draw (3, 0)  node(a02){$A_3$};
        \draw[mono] (a00)--(a01);
        \draw[mono] (a01)--(a02);
    \end{scope}
     \draw (3.8, 0) node (eq){$\xmapsto{d_0}$}; 
\begin{scope}[yshift=-1.0cm]
    \draw[thick] (4.2,1.5) rectangle (7.5,0.5);
     \begin{scope}[xshift=3cm, yshift=1cm]
        \draw (2,0)  node(a01){$A_2/A_1$};
        \draw (3.8, 0)  node(a02){$A_3/A_1$};
        \draw[mono] (a01)--(a02);
    \end{scope}
\end{scope}
    \end{tikzpicture}    
    \end{center}
 For $0\leq i\leq k$, let the $i$-th \emph{degeneracy map}
    $s_i\colon \Seq_k\cE_R\to \Seq_{k+1}(\cE_R)$ be given by doubling the $i$-th term for $i>0$ and adding a zero object $0$ for $i=0$.
    For instance, if $k=2$
    this gives
            \begin{center}
\begin{tikzpicture}
 \begin{scope}[xshift=1cm]
        \draw[thick] (0.5,0.5) rectangle (2.5,-0.5);
     \begin{scope}[xshift=0cm, yshift=0cm]
        \draw (1,0) node(a00){$A_1$};
        \draw (2,0)  node(a01){$A_2$};
        \draw[mono] (a00)--(a01);
    \end{scope}
\end{scope}
     \draw (3.8, 0) node (eq){$\xmapsto{s_i}$}; 
     \draw (6.5, 0) node (br){$
     \left\{
     \phantom{
     \begin{array}{cc}
A_2 \subseteq A_3, & i=1,\\
A_1 \subseteq A_3, & i=2,\\
A_1 \subseteq A_2, & i=3,\\
A\\
A\\
A\\
A_2/A_1 \subseteq A_3/A_1, & i=0
\end{array}
}
\right.
$};
\begin{scope}[yshift=0.1cm, xshift=1cm]
    \draw[thick] (3.5,1.5) rectangle (6.5,0.5);
     \begin{scope}[xshift=3cm, yshift=1cm]
     \draw (1,0) node (a00){$0$};
        \draw (2,0)  node(a01){$A_1$};
        \draw (3, 0)  node(a02){$A_2$};
        \draw[mono] (a00)--(a01);
        \draw[mono] (a01)--(a02);
    \end{scope}
\draw (7.3, 1) node (i1){$,\quad i=0$};   
\end{scope}
\begin{scope}[yshift=-1cm, xshift=1cm]
    \draw[thick] (3.5,1.5) rectangle (6.5,0.5);
     \begin{scope}[xshift=3cm, yshift=1cm]
     \draw (1,0) node (a00){$A_1$};
        \draw (2,0)  node(a01){$A_1$};
        \draw (3, 0)  node(a02){$A_2$};
        \draw[mono] (a00)--(a01);
        \draw[mono] (a01)--(a02);
    \end{scope}
\draw (7.3, 1) node (i1){$,\quad i=1$};   
\end{scope}
\begin{scope}[yshift=-2.1cm, xshift=1cm]
    \draw[thick] (3.5,1.5) rectangle (6.5,0.5);
     \begin{scope}[xshift=3cm, yshift=1cm]
     \draw (1,0) node (a00){$A_1$};
        \draw (2,0)  node(a01){$A_2$};
        \draw (3, 0)  node(a02){$A_2$};
        \draw[mono] (a00)--(a01);
        \draw[mono] (a01)--(a02);
    \end{scope}
\draw (7.3, 1) node (i1){$,\quad i=2$};   
\end{scope}
    \end{tikzpicture}    
    \end{center}
\end{const}

As discussed in \cite[\textsection 4]{WaldhausenTopI}, the construction $\Seq^{\mathrm{disc}}_k(\cE)$
is however not natural in the variable $k$:

\begin{rmk}
\label{discSnotSimplicial}
The assignment $[k]\mapsto \Seq^{\mathrm{disc}}_k(\cE_R)$ does \emph{not} define a simplicial set.
Given $\sigma$ in $\Seq^{\mathrm{disc}}_3(\cE_R)$
\[
\begin{tikzpicture}
\begin{scope}[yshift=-2.1cm, xshift=1cm]
    \draw[thick] (3.5,1.5) rectangle (6.5,0.5);
     \begin{scope}[xshift=3cm, yshift=1cm]
     \draw (1,0) node (a00){$A_1$};
        \draw (2,0)  node(a01){$A_2$};
        \draw (3, 0)  node(a02){$A_3$};
        \draw[mono] (a00)--(a01);
        \draw[mono] (a01)--(a02);
    \end{scope}
\end{scope}
\end{tikzpicture}
\]
many simplicial identities hold, such as
\[
d_1d_2(\sigma)=A_3= d_1d_1(\sigma).
\]
However, any expression involving the face $d_0$ is problematic. For instance, when comparing $d_0d_0$ to $d_0d_1$, the best that one can say is that there is an isomorphism - yet generally not an equality -
of the form
\begin{equation}
\label{thirdiso}
d_0d_0(\sigma)=(A_3/A_1)/(A_2/A_1)\cong A_3/A_2=d_0d_1(\sigma).
\end{equation}
\end{rmk}

Overall, the construction $\Seq_{k}^{\mathrm{disc}}(\cE_R)$ suffers from the following shortcomings:
\begin{enumerate}[leftmargin=*, label=(\arabic*)]
    \item \label{NotSimp} The assignment $[k]\mapsto \Seq^{\mathrm{disc}}_k(\cE_R)$ does not define a simplicial set of sequences in $\cE_R$, as discussed in \cref{discSnotSimplicial}.
    \item \label{Symmetry} The definition of $\Seq_{k}^{\mathrm{disc}}(\cE_R)$ breaks the natural symmetry present in the structure of $\cE_R$, in that it prioritizes subobjects over quotients.
    \item \label{Groupoids} The \emph{set} of sequences $\Seq_{k}^{\mathrm{disc}}(\cE_R)$ is not so meaningful in itself, in that it is not capable of recognizing when two sequences are \emph{isomorphic}, and should instead be replaced with the \emph{groupoid} or \emph{space} of sequences.
\end{enumerate}
A way to address at least some of these issues is to consider the following variant (which can be seen as a variant of \cite[\textsection 1]{waldhausen} and as an instance of \cite[\textsection 4]{WaldhausenTopI}, taking in $\cE_R$ cofibrations to be monomorphisms and weak equivalences to be isomorphisms):

\begin{const}
    For $k\geq0$, let $\Seq_k(\cE_R)$ denote the groupoid of all sequences of $n$ subobjects
    in $\cE_R$; that is, the maximal groupoid in the category of functors from $[k]$ to $\cE_R$ spanned by the set of objects $\Seq^{\mathrm{disc}}_k(\cE_R)$.
    For $0\leq i\leq k$, let the $i$-th \emph{face map} $d_i\colon \Seq_k(\cE_R)\to \Seq_{k-1}(\cE_R)$ and the $i$-th \emph{degeneracy map} $s_i\colon \Seq_k(\cE_R)\to \Seq_{k+1}(\cE_R)$ to be defined by extending the formulas from \cref{discSeq} in the obvious way. 
\end{const}

By making appropriately coherent the isomorphism witnessing the failure of the simplicial identities for $\Seq_k(\cE_R)$ -- e.g.~the one from \eqref{thirdiso} -- one could show:

\begin{prop}
The assignment $[k]\mapsto \Seq_k(\cE_R) $ defines a pseudo simplicial groupoid.
\end{prop}

The construction $\Seq_k(\cE_R)$ addresses \ref{Groupoids}, but it only partially addresses \ref{NotSimp} --
in that it involves a higher categorical metatheory -- and it does not address \ref{Symmetry}.

\subsubsection{The $S_\bullet$-construction}

An attempt at solving the issues \ref{NotSimp}, \ref{Groupoids} and \ref{Symmetry} going in a different direction is the following (cf.~\cite[\textsection 1.4]{waldhausen}):

Denote by $\Ar[k]=[k]^{[1]}$ the \emph{arrow category} of $[k]$; that is, the category of functors $[1]\to[k]$ and natural transformations.

\begin{const}
\label{Sdisc}
For $n\geq0$, let $S^{\mathrm{disc}}_k(\cE_R)$ denote the set of \emph{exact functors} $\Ar[k]\to\cE_R$.
That is, of functors $\Ar[k]\to\cE_R$ which:
\begin{itemize}[leftmargin=*]
\item send the objects of the form $(i,i)$ to a zero object in $\cE_R$,
\item send the morphisms of the form $(i,j)\to(i,j+\ell)$ to  monomorphisms in $\cE_R$,
\item send the morphisms $(i,j)\to(i+k,j)$ to epimorphisms in $\cE_R$,
\item sends the commutative squares involving two maps of the form $(i,j)\to (i+k,j)$  and two maps of the form $(i,j)\to (i,j+\ell)$ to bicartesian squares in $\cE_R$.
\end{itemize}
For instance, a functor $\Ar[3]\to\cE_R$ defining an object of $S^{\mathrm{disc}}_3(\cE_R)$ is of the form
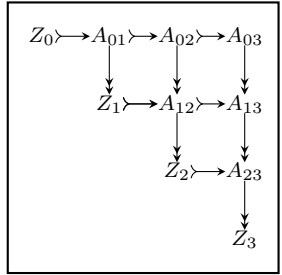
\begin{center}
\begin{tikzpicture}[scale=0.9, inner sep=0pt, font=\footnotesize, baseline=-30pt]
\begin{scope}[xshift=0.0cm]
\draw[thick] (0.5,0.5) rectangle (4.5,-3.5);
\begin{scope}
    \draw (1,0) node(a00){{\footnotesize $Z_{0}$}};
\draw (2,0) node(a01){{\footnotesize $A_{01}$}};
\draw (2,-1) node(a11) {{\footnotesize $Z_{1}$}};
\draw  (3, -1)   node(a12){{\footnotesize $A_{12}$}};
\draw  (3, 0)   node(a02){{\footnotesize $A_{02}$}};
\draw (3,-2) node (a22){{\footnotesize $Z_{2}$}};
\draw  (4, 0)   node (a03){{\footnotesize $A_{03}$}};
\draw (4, -1)   node (a13){{\footnotesize $A_{13}$}};
\draw (4, -2)   node (a23){{\footnotesize $A_{23}$}};
\draw (4, -3) node (a33){{\footnotesize $Z_{3}$}};
\draw[mono] (a11)--(a12);
\draw[mono] (a12)--(a13);
\draw[mono] (a22)--(a23);
\draw[mono] (a00)--(a01);
\draw[mono] (a11)--(a12);
\draw[mono] (a01)--(a02);
\draw[mono] (a02)--(a03);
\draw[epi] (a01)--(a11);
\draw[epi] (a12)--(a22);
\draw[epi] (a02)--(a12);
\draw[epi] (a03)--(a13);
\draw[epi] (a13)--(a23);
\draw[epi] (a23)--(a33);
\end{scope}
\end{scope}
\end{tikzpicture}
\end{center}
with the convention that all objects featuring on the diagonal are zero objects, all horizontal morphisms are monomorphisms, all vertical morphisms are epimorphisms, and all squares are bicartesian squares.
    For $0\leq i\leq k$, let the $i$-th \emph{face map} $d_i\colon S^{\mathrm{disc}}_k(\cE_R)\to S^{\mathrm{disc}}_{k-1}(\cE_R)$ be given by skipping the $i$-th row and the $i$-th column. For instance, if $k=3$ the face maps $d_2$ and $d_3$ act as:
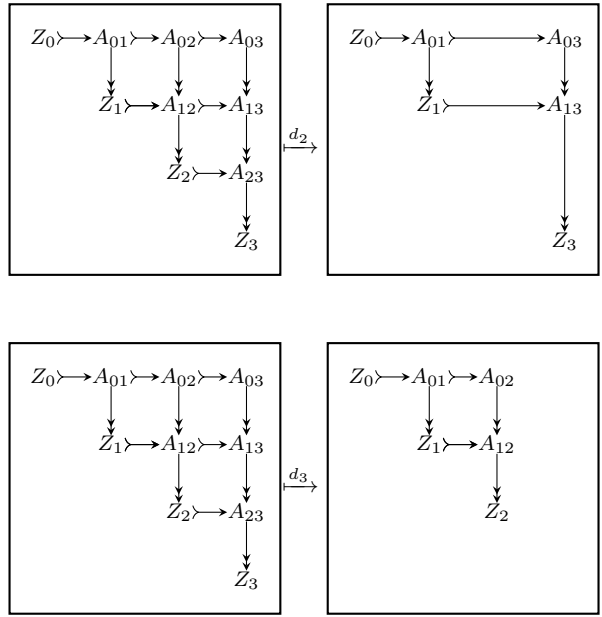
\begin{center}
\begin{tikzpicture}[scale=0.9, inner sep=0pt, font=\footnotesize, baseline=-30pt]
\def\hsh{0cm}
\begin{scope}[xshift=0.0cm]
\draw[thick] (0.5,0.5) rectangle (4.5,-3.5);
\begin{scope}
    \draw (1,0) node(a00){{\footnotesize $Z_{0}$}};
\draw (2,0) node(a01){{\footnotesize $A_{01}$}};
\draw (2,-1) node(a11) {{\footnotesize $Z_{1}$}};
\draw  (3, -1)   node(a12){{\footnotesize $A_{12}$}};
\draw  (3, 0)   node(a02){{\footnotesize $A_{02}$}};
\draw (3,-2) node (a22){{\footnotesize $Z_{2}$}};
\draw  (4, 0)   node (a03){{\footnotesize $A_{03}$}};
\draw (4, -1)   node (a13){{\footnotesize $A_{13}$}};
\draw (4, -2)   node (a23){{\footnotesize $A_{23}$}};
\draw (4, -3) node (a33){{\footnotesize $Z_{3}$}};
\draw[mono] (a11)--(a12);
\draw[mono] (a12)--(a13);
\draw[mono] (a22)--(a23);
\draw[mono] (a00)--(a01);
\draw[mono] (a11)--(a12);
\draw[mono] (a01)--(a02);
\draw[mono] (a02)--(a03);
\draw[epi] (a01)--(a11);
\draw[epi] (a12)--(a22);
\draw[epi] (a02)--(a12);
\draw[epi] (a03)--(a13);
\draw[epi] (a13)--(a23);
\draw[epi] (a23)--(a33);
\end{scope}
\end{scope}
\draw (4.8,-1.5) node (si){$\xmapsto{d_2}$};
\begin{scope}[xshift=4.7cm]
\draw[thick] (0.5,0.5) rectangle (4.5,-3.5);
\begin{scope}
    \draw (1,0) node(a00){{\footnotesize $Z_{0}$}};
\draw (2,0) node(a01){{\footnotesize $A_{01}$}};
\draw (2,-1) node(a11) {{\footnotesize $Z_{1}$}};
\draw  (4, 0)   node (a03){{\footnotesize $A_{03}$}};
\draw (4, -1)   node (a13){{\footnotesize $A_{13}$}};
\draw (4, -3) node (a33){{\footnotesize $Z_{3}$}};
\draw[mono] (a11)--(a13);
\draw[mono] (a00)--(a01);
\draw[mono] (a01)--(a03);
%
\draw[epi] (a01)--(a11);
\draw[epi] (a03)--(a13);
\draw[epi] (a13)--(a33);
\end{scope}
\end{scope}
\end{tikzpicture}
\end{center}

\begin{center}
\begin{tikzpicture}[scale=0.9, inner sep=0pt, font=\footnotesize, baseline=-30pt]
\def\hsh{0cm}
\begin{scope}[xshift=\hsh]
\draw[thick] (0.5,0.5) rectangle (4.5,-3.5);
\begin{scope}
    \draw (1,0) node(a00){{\footnotesize $Z_{0}$}};
\draw (2,0) node(a01){{\footnotesize $A_{01}$}};
\draw (2,-1) node(a11) {{\footnotesize $Z_{1}$}};
\draw  (3, -1)   node(a12){{\footnotesize $A_{12}$}};
\draw  (3, 0)   node(a02){{\footnotesize $A_{02}$}};
\draw (3,-2) node (a22){{\footnotesize $Z_{2}$}};
\draw  (4, 0)   node (a03){{\footnotesize $A_{03}$}};
\draw (4, -1)   node (a13){{\footnotesize $A_{13}$}};
\draw (4, -2)   node (a23){{\footnotesize $A_{23}$}};
\draw (4, -3) node (a33){{\footnotesize $Z_{3}$}};
\draw[mono] (a11)--(a12);
\draw[mono] (a12)--(a13);
\draw[mono] (a22)--(a23);
\draw[mono] (a00)--(a01);
\draw[mono] (a11)--(a12);
\draw[mono] (a01)--(a02);
\draw[mono] (a02)--(a03);
\draw[epi] (a01)--(a11);
\draw[epi] (a12)--(a22);
\draw[epi] (a02)--(a12);
\draw[epi] (a03)--(a13);
\draw[epi] (a13)--(a23);
\draw[epi] (a23)--(a33);
\end{scope}
\end{scope}
\draw (4.8cm+\hsh,-1.5) node (si){$\xmapsto{d_3}$};
\begin{scope}[xshift=4.7cm+\hsh]
\draw[thick] (0.5,0.5) rectangle (4.5,-3.5);
\begin{scope}
  \draw (1,0) node(a00){{\footnotesize $Z_{0}$}};
\draw (2,0) node(a01){{\footnotesize $A_{01}$}};
\draw (2,-1) node(a11) {{\footnotesize $Z_{1}$}};
\draw  (3, -1)   node(a12){{\footnotesize $A_{12}$}};
\draw  (3, 0)   node(a02){{\footnotesize $A_{02}$}};
\draw (3,-2) node (a22){{\footnotesize $Z_{2}$}};
\draw[mono] (a11)--(a12);
\draw[mono] (a00)--(a01);
\draw[mono] (a11)--(a12);
\draw[mono] (a01)--(a02);
%
\draw[epi] (a01)--(a11);
\draw[epi] (a12)--(a22);
\draw[epi] (a02)--(a12);
\end{scope}
\end{scope}
\end{tikzpicture}
\end{center}

For $0\leq i\leq k$, let
    the $i$-th \emph{degeneracy map} $s_i\colon S^{\mathrm{disc}}_k(\cE_R)\to S^{\mathrm{disc}}_{k+1}(\cE_R)$ be given by by doubling the $i$-th term. For instance, if $k=2$ the degeneracy maps $s_0$ and $s_1$ are as follows
    
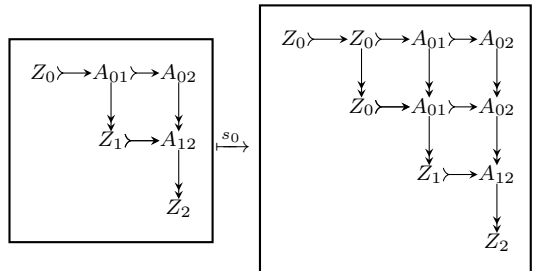
\begin{center}
        \begin{tikzpicture}[scale=0.9, inner sep=0pt, font=\footnotesize, baseline=-30pt]
\begin{scope}[xshift=0.0cm]
\draw[thick] (1.5,0.0) rectangle (4.5,-3);
\begin{scope}[xshift=1cm, yshift=-0.5cm]
  \draw (1,0) node(a00){{\footnotesize $Z_{0}$}};
\draw (2,0) node(a01){{\footnotesize $A_{01}$}};
\draw (2,-1) node(a11) {{\footnotesize $Z_{1}$}};
\draw  (3, -1)   node(a12){{\footnotesize $A_{12}$}};
\draw  (3, 0)   node(a02){{\footnotesize $A_{02}$}};
\draw (3,-2) node (a22){{\footnotesize $Z_{2}$}};
\draw[mono] (a11)--(a12);
\draw[mono] (a00)--(a01);
\draw[mono] (a11)--(a12);
\draw[mono] (a01)--(a02);
%
\draw[epi] (a01)--(a11);
\draw[epi] (a12)--(a22);
\draw[epi] (a02)--(a12);
\end{scope}
\end{scope}
\draw (4.8,-1.5) node (si){$\xmapsto{s_0}$};
\begin{scope}[xshift=4.7cm]
\draw[thick] (0.5,0.5) rectangle (4.5,-3.5);
\begin{scope}
    \draw (1,0) node(a00){{\footnotesize $Z_{0}$}};
\draw (2,0) node(a01){{\footnotesize $Z_{0}$}};
\draw (2,-1) node(a11) {{\footnotesize $Z_{0}$}};
\draw  (3, -1)   node(a12){{\footnotesize $A_{01}$}};
\draw  (3, 0)   node(a02){{\footnotesize $A_{01}$}};
\draw (3,-2) node (a22){{\footnotesize $Z_{1}$}};
\draw  (4, 0)   node (a03){{\footnotesize $A_{02}$}};
\draw (4, -1)   node (a13){{\footnotesize $A_{02}$}};
\draw (4, -2)   node (a23){{\footnotesize $A_{12}$}};
\draw (4, -3) node (a33){{\footnotesize $Z_{2}$}};
\draw[mono] (a11)--(a12);
\draw[mono] (a12)--(a13);
\draw[mono] (a22)--(a23);
\draw[mono] (a00)--(a01);
\draw[mono] (a11)--(a12);
\draw[mono] (a01)--(a02);
\draw[mono] (a02)--(a03);
\draw[epi] (a01)--(a11);
\draw[epi] (a12)--(a22);
\draw[epi] (a02)--(a12);
\draw[epi] (a03)--(a13);
\draw[epi] (a13)--(a23);
\draw[epi] (a23)--(a33);
\end{scope}
\end{scope}
\end{tikzpicture}
    \end{center}
        
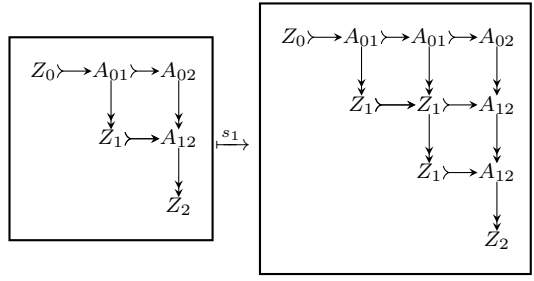
\begin{center}
        \begin{tikzpicture}[scale=0.9, inner sep=0pt, font=\footnotesize, baseline=-30pt]
\begin{scope}[xshift=0.0cm]
\draw[thick] (1.5,0.0) rectangle (4.5,-3);
\begin{scope}[xshift=1cm, yshift=-0.5cm]
  \draw (1,0) node(a00){{\footnotesize $Z_{0}$}};
\draw (2,0) node(a01){{\footnotesize $A_{01}$}};
\draw (2,-1) node(a11) {{\footnotesize $Z_{1}$}};
\draw  (3, -1)   node(a12){{\footnotesize $A_{12}$}};
\draw  (3, 0)   node(a02){{\footnotesize $A_{02}$}};
\draw (3,-2) node (a22){{\footnotesize $Z_{2}$}};
\draw[mono] (a11)--(a12);
\draw[mono] (a00)--(a01);
\draw[mono] (a11)--(a12);
\draw[mono] (a01)--(a02);
%
\draw[epi] (a01)--(a11);
\draw[epi] (a12)--(a22);
\draw[epi] (a02)--(a12);
\end{scope}
\end{scope}
\draw (4.8,-1.5) node (si){$\xmapsto{s_1}$};
\begin{scope}[xshift=4.7cm]
\draw[thick] (0.5,0.5) rectangle (4.5,-3.5);
\begin{scope}
    \draw (1,0) node(a00){{\footnotesize $Z_{0}$}};
\draw (2,0) node(a01){{\footnotesize $A_{01}$}};
\draw (2,-1) node(a11) {{\footnotesize $Z_{1}$}};
\draw  (3, -1)   node(a12){{\footnotesize $Z_{1}$}};
\draw  (3, 0)   node(a02){{\footnotesize $A_{01}$}};
\draw (3,-2) node (a22){{\footnotesize $Z_{1}$}};
\draw  (4, 0)   node (a03){{\footnotesize $A_{02}$}};
\draw (4, -1)   node (a13){{\footnotesize $A_{12}$}};
\draw (4, -2)   node (a23){{\footnotesize $A_{12}$}};
\draw (4, -3) node (a33){{\footnotesize $Z_{2}$}};
\draw[mono] (a11)--(a12);
\draw[mono] (a12)--(a13);
\draw[mono] (a22)--(a23);
\draw[mono] (a00)--(a01);
\draw[mono] (a11)--(a12);
\draw[mono] (a01)--(a02);
\draw[mono] (a02)--(a03);
\draw[epi] (a01)--(a11);
\draw[epi] (a12)--(a22);
\draw[epi] (a02)--(a12);
\draw[epi] (a03)--(a13);
\draw[epi] (a13)--(a23);
\draw[epi] (a23)--(a33);
\end{scope}
\end{scope}
\end{tikzpicture}
    \end{center}
\end{const}

One can then check that the construction $S^{\mathrm{disc}}_k(\cE)$ is natural in $k$:

\begin{prop}
The assignment $[k]\mapsto S_k^{\mathrm{disc}}(\cE_R) $ defines a simplicial set $S_\bullet^{\mathrm{disc}}(\cE_R)$.
\end{prop}

The construction $S_k^{\mathrm{disc}}(\cE_R)$ solves \ref{NotSimp} without requiring higher categorical metatheory, and it solves \ref{Symmetry}, but does not fix \ref{Groupoids}. So one can then mix the two adjustments $\Seq_k(\cE_R)$ and $S_k^{\mathrm{disc}}(\cE_R)$, culminating with Waldhausen's original $S_\bullet$-construction (cf.~\cite[\textsection 1.3]{waldhausen} in the special case that weak equivalences are chosen to be isomorphisms):

\begin{const}
    For $k\geq0$, let $S_k(\cE_R)$ denote the groupoid of exact functors from $\Ar[k]$ to $\cE_R$; that is, the maximal groupoid in the category of functors $\Ar[k]$ to $\cE_R$ spanned by the set of objects $S^{\mathrm{disc}}_k(\cE_R)$.
For $0\leq i\leq k$,
let the $i$-th \emph{face map} $d_i\colon S_k(\cE_R)\to S_{k-1}(\cE_R)$ and the $i$-th \emph{degeneracy map} $s_i\colon S_k(\cE_R)\to S_{k+1}(\cE_R)$ be
defined by extending the formulas from \cref{Sdisc} in the obvious way.
\end{const}

One can then check that this produces a simplicial object (cf.~\cite[\textsection 1.3]{waldhausen}), which we refer to as the \emph{$S_\bullet$-construction} of $\cE_R$:

\begin{prop}
The assignment $[k]\mapsto S_k(\cE_R) $ defines a simplicial groupoid $S_\bullet(\cE_R)$.
\end{prop}

The following proposition explains how the construction $S_{\bullet}(\cE_R)$ can really be understood as an appropriate correction of $\Seq_{\bullet}(\cE_R)$, in that (cf.~\cite[Lemma 2.4.9]{DKbook}):

\begin{prop}
\label{SandSeq}
For $k\geq0$, the canonical projection on ``the $0$-th row'' is an equivalence of groupoids
\[S_k(\cE_R)\xrightarrow{\simeq}\Seq_k(\cE_R).\]
\end{prop}

\begin{proof}[Idea]
For instance, for $k=3$ this map of groupoids can be depicted as
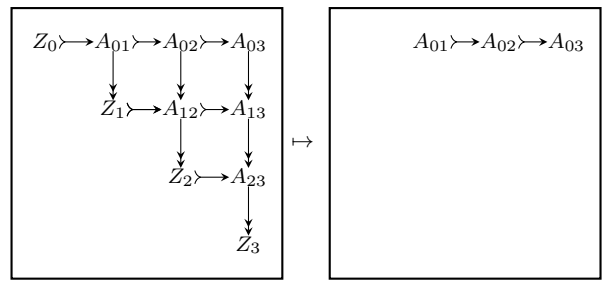
\begin{center}
\begin{tikzpicture}[scale=0.9, inner sep=0pt, font=\footnotesize, baseline=-30pt]
\begin{scope}[xshift=0.0cm]
\draw[thick] (0.5,0.5) rectangle (4.5,-3.5);
\begin{scope}
    \draw (1,0) node(a00){{\footnotesize $Z_{0}$}};
\draw (2,0) node(a01){{\footnotesize $A_{01}$}};
\draw (2,-1) node(a11) {{\footnotesize $Z_{1}$}};
\draw  (3, -1)   node(a12){{\footnotesize $A_{12}$}};
\draw  (3, 0)   node(a02){{\footnotesize $A_{02}$}};
\draw (3,-2) node (a22){{\footnotesize $Z_{2}$}};
\draw  (4, 0)   node (a03){{\footnotesize $A_{03}$}};
\draw (4, -1)   node (a13){{\footnotesize $A_{13}$}};
\draw (4, -2)   node (a23){{\footnotesize $A_{23}$}};
\draw (4, -3) node (a33){{\footnotesize $Z_{3}$}};
\draw[mono] (a11)--(a12);
\draw[mono] (a12)--(a13);
\draw[mono] (a22)--(a23);
\draw[mono] (a00)--(a01);
\draw[mono] (a11)--(a12);
\draw[mono] (a01)--(a02);
\draw[mono] (a02)--(a03);
\draw[epi] (a01)--(a11);
\draw[epi] (a12)--(a22);
\draw[epi] (a02)--(a12);
\draw[epi] (a03)--(a13);
\draw[epi] (a13)--(a23);
\draw[epi] (a23)--(a33);
\end{scope}
\end{scope}
\draw (4.8,-1.5) node (si){$\mapsto$};
\begin{scope}[xshift=4.7cm]
\draw[thick] (0.5,0.5) rectangle (4.5,-3.5);
\begin{scope}
\draw (2,0) node(a01){{\footnotesize $A_{01}$}};
\draw  (3, 0)   node(a02){{\footnotesize $A_{02}$}};
\draw  (4, 0)   node (a03){{\footnotesize $A_{03}$}};
%
\draw[mono] (a01)--(a02);
\draw[mono] (a02)--(a03);
\end{scope}
\end{scope}
\end{tikzpicture}
\end{center}
By completing including the zero module in a few positions of the diagram and completing all the relevant spans to pushout squares in $\cE_R$, which can be done is an essentially unique way, one can consider the assignment
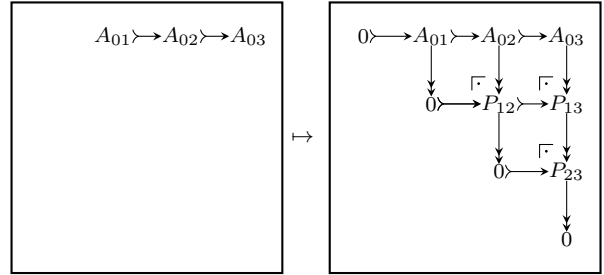
\begin{center}
\begin{tikzpicture}[scale=0.9, inner sep=0pt, font=\footnotesize, baseline=-30pt]
\begin{scope}[xshift=0.0cm]
\draw[thick] (0.5,0.5) rectangle (4.5,-3.5);
\begin{scope}
\draw (2,0) node(a01){{\footnotesize $A_{01}$}};
\draw  (3, 0)   node(a02){{\footnotesize $A_{02}$}};
\draw  (4, 0)   node (a03){{\footnotesize $A_{03}$}};
%
\draw[mono] (a01)--(a02);
\draw[mono] (a02)--(a03);
\end{scope}
\end{scope}
\draw (4.8,-1.5) node (si){$\mapsto$};
\begin{scope}[xshift=4.7cm]
\draw[thick] (0.5,0.5) rectangle (4.5,-3.5);
\begin{scope}
    \draw (1,0) node(a00){{\footnotesize $0$}};
\draw (2,0) node(a01){{\footnotesize $A_{01}$}};
\draw (2,-1) node(a11) {{\footnotesize $0$}};
\draw  (3, -1)   node(a12){{\footnotesize $P_{12}$}};
\draw  (3, 0)   node(a02){{\footnotesize $A_{02}$}};
\draw (3,-2) node (a22){{\footnotesize $0$}};
\draw  (4, 0)   node (a03){{\footnotesize $A_{03}$}};
\draw (4, -1)   node (a13){{\footnotesize $P_{13}$}};
\draw (4, -2)   node (a23){{\footnotesize $P_{23}$}};
\draw (4, -3) node (a33){{\footnotesize $0$}};
\draw[mono] (a11)--(a12);
\draw[mono] (a12)--(a13);
\draw[mono] (a22)--(a23);
\draw[mono] (a00)--(a01);
\draw[mono] (a11)--(a12);
\draw[mono] (a01)--(a02);
\draw[mono] (a02)--(a03);
\draw[epi] (a01)--(a11);
\draw[epi] (a12)--(a22);
\draw[epi] (a02)--(a12);
\draw[epi] (a03)--(a13);
\draw[epi] (a13)--(a23);
\draw[epi] (a23)--(a33);
\begin{scope}[shift=($(a02)!.6!(a13)$), scale=0.4]
\draw +(.5,0) -- +(0,0)  -- +(0,-0.5);
\fill +(.25,-.25) circle (.05);
\end{scope}
\begin{scope}[shift=($(a01)!.6!(a12)$), scale=0.4]
\draw +(.5,0) -- +(0,0)  -- +(0,-0.5);
\fill +(.25,-.25) circle (.05);
\end{scope}
\begin{scope}[shift=($(a12)!.6!(a23)$), scale=0.4]
\draw +(.5,0) -- +(0,0)  -- +(0,-0.5);
\fill +(.25,-.25) circle (.05);
\end{scope}
\end{scope}
\end{scope}
\end{tikzpicture}
\end{center}
and verify it defines an inverse equivalence of groupoids
\[\Seq_k(\cE_R)\xrightarrow{\simeq}S_k(\cE_R).\]
for the desired map. The statement for general $k$ can be proven through an appropriate enhancement of this argument.
\end{proof}

We can take a closer look at how, besides for the sequences from \cref{SandSeq}, many other relevant ``homological'' features of $R$ are retained in $S_\bullet(\cE_R)$:

\begin{rmk}
We have that:
 \begin{enumerate}[leftmargin=*]
 \setcounter{enumi}{-1}
    \item If $\mathrm{Z}(\cE_R)$ denotes
    the groupoid of zero objects in $\cE_R$, there is an isomorphism of groupoids 
     \[S_0(\cE_R)\cong \mathrm{Z}(\cE_R).\]
     \item If $\Ob(\cE_R)$ denotes the core of $\cE_R$, i.e., the maximal groupoid inside $\cE_R$, as a special case of \cref{SandSeq} there is an equivalence of groupoids
     \[S_1(\cE_R)\simeq\Ob(\cE_R).\]
     \item If $\mathrm{ExSeq}(\cE_R)$
     denotes the groupoid
     of short exact sequences in $\cE_R$, there is an equivalence of groupoids
     \[S_2(\cE_R)\simeq \mathrm{ExSeq}(\cE_R),\]
 which can be depicted as
                
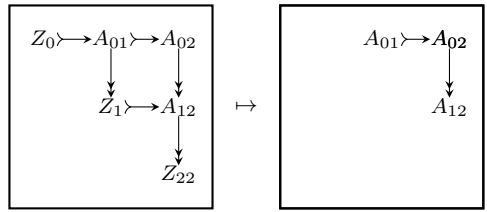
\begin{center}
\begin{tikzpicture}[scale=0.9, inner sep=0pt, font=\footnotesize, baseline=-30pt]
\begin{scope}
\draw[thick] (0.5,0.5) rectangle (3.5,-2.5);
\begin{scope}
    \draw (1,0) node(a00){{\footnotesize $Z_{0}$}};
\draw (2,0) node(a01){{\footnotesize $A_{01}$}};
\draw (2,-1) node(a11) {{\footnotesize $Z_{1}$}};
\draw  (3, -1)   node(a12){{\footnotesize $A_{12}$}};
\draw  (3, 0)   node(a02){{\footnotesize $A_{02}$}};
\draw (3,-2) node (a22){{\footnotesize $Z_{22}$}};
\draw[mono] (a00)--(a01);
\draw[mono] (a11)--(a12);
\draw[mono] (a01)--(a02);
\draw[epi] (a01)--(a11);
\draw[epi] (a12)--(a22);
\draw[epi] (a02)--(a12);

\end{scope}
\end{scope}
\draw (4.0,-1) node (si){$\mapsto$};
\begin{scope}[xshift=4.0cm]
\draw[thick] (0.5,0.5) rectangle (3.5,-2.5);
\begin{scope}
\draw (2,0) node(a01){{\footnotesize $A_{01}$}};
\draw  (3, 0)   node(a02){{\footnotesize $A_{02}$}};
\draw[mono] (a01)--(a02);
\end{scope}
\end{scope}
\begin{scope}[xshift=4.0cm]
\draw[thick] (0.5,0.5) rectangle (3.5,-2.5);
\begin{scope}
\draw  (3, 0)   node(a02){{\footnotesize $A_{02}$}};
\draw  (3, -1)   node(a12){{\footnotesize $A_{12}$}};
\draw[epi] (a02)--(a12);
\end{scope}
\end{scope}
\end{tikzpicture}
\end{center}

\item If $\mathrm{BiCartSq}(\cE_R)$ denotes the groupoid of bicartesian squares involving two monomorphisms and two epimorphisms in $\cE_R$, there is an equivalence of groupoids
\[S_3(\cE_R)\simeq \mathrm{BiCartSq}(\cE_R),\]
which can be depicted as
    
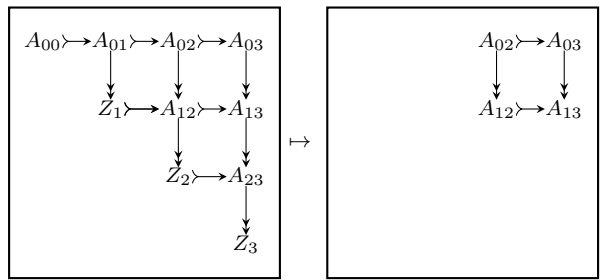
\begin{center}
\begin{tikzpicture}[scale=0.9, inner sep=0pt, font=\footnotesize, baseline=-30pt]
\begin{scope}[xshift=0.0cm]
\draw[thick] (0.5,0.5) rectangle (4.5,-3.5);
\begin{scope}
    \draw (1,0) node(a00){{\footnotesize $A_{00}$}};
\draw (2,0) node(a01){{\footnotesize $A_{01}$}};
\draw (2,-1) node(a11) {{\footnotesize $Z_{1}$}};
\draw  (3, -1)   node(a12){{\footnotesize $A_{12}$}};
\draw  (3, 0)   node(a02){{\footnotesize $A_{02}$}};
\draw (3,-2) node (a22){{\footnotesize $Z_{2}$}};
\draw  (4, 0)   node (a03){{\footnotesize $A_{03}$}};
\draw (4, -1)   node (a13){{\footnotesize $A_{13}$}};
\draw (4, -2)   node (a23){{\footnotesize $A_{23}$}};
\draw (4, -3) node (a33){{\footnotesize $Z_{3}$}};
\draw[mono] (a11)--(a12);
\draw[mono] (a12)--(a13);
\draw[mono] (a22)--(a23);
\draw[mono] (a00)--(a01);
\draw[mono] (a11)--(a12);
\draw[mono] (a01)--(a02);
\draw[mono] (a02)--(a03);
\draw[epi] (a01)--(a11);
\draw[epi] (a12)--(a22);
\draw[epi] (a02)--(a12);
\draw[epi] (a03)--(a13);
\draw[epi] (a13)--(a23);
\draw[epi] (a23)--(a33);
\end{scope}
\end{scope}
\draw (4.8,-1.5) node (si){$\mapsto$};
\begin{scope}[xshift=4.7cm]
\draw[thick] (0.5,0.5) rectangle (4.5,-3.5);
\begin{scope}
\draw  (3, -1)   node(a12){{\footnotesize $A_{12}$}};
\draw  (3, 0)   node(a02){{\footnotesize $A_{02}$}};
\draw  (4, 0)   node (a03){{\footnotesize $A_{03}$}};
\draw (4, -1)   node (a13){{\footnotesize $A_{13}$}};
\draw[mono] (a12)--(a13);
\draw[mono] (a02)--(a03);
\draw[epi] (a02)--(a12);
\draw[epi] (a03)--(a13);
\end{scope}
\end{scope}
\end{tikzpicture}
\end{center}
with inverse equivalence given by
               
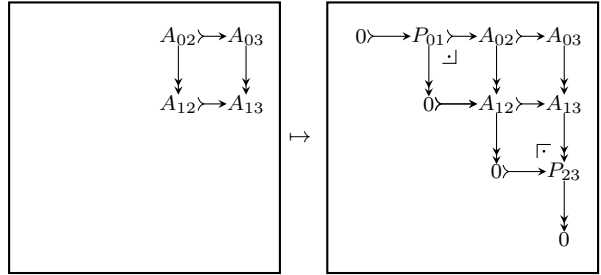
\begin{center}
\begin{tikzpicture}[scale=0.9, inner sep=0pt, font=\footnotesize, baseline=-30pt]
\begin{scope}[xshift=0.0cm]
\draw[thick] (0.5,0.5) rectangle (4.5,-3.5);
\begin{scope}
\draw  (3, -1)   node(a12){{\footnotesize $A_{12}$}};
\draw  (3, 0)   node(a02){{\footnotesize $A_{02}$}};
\draw  (4, 0)   node (a03){{\footnotesize $A_{03}$}};
\draw (4, -1)   node (a13){{\footnotesize $A_{13}$}};
\draw[mono] (a12)--(a13);
\draw[mono] (a02)--(a03);
\draw[epi] (a02)--
(a12);
\draw[epi] (a03)--(a13);
\end{scope}
\end{scope}
\draw (4.8,-1.5) node (si){$\mapsto$};
\begin{scope}[xshift=4.7cm]
\draw[thick] (0.5,0.5) rectangle (4.5,-3.5);
\begin{scope}
\draw  (3, -1)   node(a12){{\footnotesize $A_{12}$}};
\draw  (3, 0)   node(a02){{\footnotesize $A_{02}$}};
\draw  (4, 0)   node (a03){{\footnotesize $A_{03}$}};
\draw (4, -1)   node (a13){{\footnotesize $A_{13}$}};
    \draw (1,0) node(a00){{\footnotesize $0$}};
\draw (2,0) node(a01){{\footnotesize $P_{01}$}};
\draw (2,-1) node(a11) {{\footnotesize $0$}};
\draw (3,-2) node (a22){{\footnotesize $0$}};
\draw (4, -2)   node (a23){{\footnotesize $P_{23}$}};
\draw (4, -3) node (a33){{\footnotesize $0$}};
\draw[mono] (a11)--(a12);
\draw[mono] (a12)--(a13);
\draw[mono] (a22)--(a23);
\draw[mono] (a00)--(a01);
\draw[mono] (a11)--(a12);
\draw[mono] (a01)--(a02);
\draw[mono] (a02)--(a03);
\draw[epi] (a01)--(a11);
\draw[epi] (a12)--(a22);
\draw[epi] (a02)--
(a12);
\draw[epi] (a03)--(a13);
\draw[epi] (a13)--(a23);
\draw[epi] (a23)--(a33);
\begin{scope}[shift=($(a12)!.6!(a23)$), scale=0.4]
\draw +(.5,0) -- +(0,0)  -- +(0,-0.5);
\fill +(.25,-.25) circle (.05);
\end{scope}
\begin{scope}[shift=($(a01)!.4!(a12)$), scale=0.4]
\draw +(-.5,0) -- +(0,0)  -- +(0,.5);
\fill +(-.25,.25) circle (.05);
\end{scope}
\end{scope}
\end{scope}
\end{tikzpicture}
\end{center}
\end{enumerate}
\end{rmk}

From $S_\bullet(\cE_R)$, one can define the $K$-theory space $K(\cE_R)$ (cf.~\cite[\textsection 1.3]{waldhausen}), by further applying $\mathrm{Real}\colon\gpd^{\Delta^{\op}}\to\cS^{\Delta^{\op}}\to\cS$, the realization of simplicial groupoids through simplicial spaces:

\begin{const} 
The \emph{$K$-theory space} of $\cE_R$ is
\[
K(\cE_R) \coloneqq \Omega(\mathrm{Real} (S_\bullet(\cE_R))).\]
\end{const}

\begin{rmk}
\label{KtheoryRing}
If $\cE_R$ is the category of modules over $R$, then $K(\cE_R)$ is contractible (see e.g.~\cite[\textsection II.2]{WeibelKBook}).
Instead, if $\cE_R$ is the category of finitely generated projective modules over $R$, then $K(\cE_R)$ is typically interesting, and it is what one refers to as the $K$-theory space $K(R)$ of the ring $R$. The homotopy groups of $K(R)$ have been fully computed for the case of $R$ being a finite field (see \cite[Theorem 8]{QuillenCohomologyK}),
and are only partially known for $R=\mathbb Z$ (see \cite{WeibelIntegers}).
\end{rmk}

We will see later in \cref{Section2} how the space $K(\cE_R)$ can be delooped infinitely many times, so it actually comes from an $\Omega$-spectrum.

\subsection{Generalized $S_\bullet$-construction}

We now mention the generalizations of Waldhausen's construction that have been studied in the literature, culminating with a detailed description in \cref{SdotForSigma} for a rather general one.

\subsubsection{Pointed stable contexts}

\label{stable}

We collect here all extensions of Waldhausen's $S_\bullet$-construction to a generalized input $\cE$ that maintains a version of the properties of ``stability'' and ``pointedness'' (or more generally ``augmentation'').
The structure of $\cE$ needs to include a class of objects, a class of distinguished objects, two classes of morphisms -- horizontal and vertical -- and a class of squares. The properties required of $\cE$ roughly encode the following, to be understood in a strict or weak sense, depending on the context:
\begin{itemize}[leftmargin=*, label=$\diamond$]
    \item {\bf pasting}: the fact that the horizontal morphisms can be composed, vertical morphisms can be composed, and squares can be composed both horizontally and vertically;
    \item {\bf pointedness} (or more generally {\bf augmentation}): the fact that there is a distinguished class of objects -- typically the class of zero objects of some sort -- with the property that every object receives uniquely a horizontal morphism from an object in this distinguished class, and every object maps uniquely through a vertical morphism to an object in this distinguished class;
    \item {\bf stability}: the fact that every square is uniquely determined by the span in its boundary and also uniquely determined by the cospan in its boundary, through the depicted assignments:
\begin{equation}
\label{SpanCospan}
\begin{tikzpicture}[scale=0.9, inner sep=0pt, font=\footnotesize, baseline=-15pt]
\begin{scope}[xshift=0.0cm]
\draw[thick] (2.5,0.5) rectangle (4.5,-1.5);
\begin{scope}
\draw  (3, -1)   node(a12){{\footnotesize $A_{12}$}};
\draw  (3, 0)   node(a02){{\footnotesize $A_{02}$}};
\draw  (4, 0)   node (a03){{\footnotesize $A_{03}$}};
\draw (4, -1)   node (a13){{\footnotesize $A_{13}$}};
\draw[mono] (a12)--(a13);
\draw[mono] (a02)--(a03);
\draw[epi] (a02)--(a12);
\draw[epi] (a03)--(a13);
\end{scope}
\end{scope}
\draw (4.8,-0.5) node (si){$\mapsto$};
\draw (2.1,-0.5) node (si){\rotatebox{180}{$\mapsto$}};
\begin{scope}[xshift=0.0cm]
\draw[thick] (2.5,0.5) rectangle (4.5,-1.5);
\begin{scope}
\draw  (3, -1)   node(a12){{\footnotesize $A_{12}$}};
\draw  (3, 0)   node(a02){{\footnotesize $A_{02}$}};
\draw  (4, 0)   node (a03){{\footnotesize $A_{03}$}};
\draw (4, -1)   node (a13){{\footnotesize $A_{13}$}};
\draw[mono] (a12)--(a13);
\draw[mono] (a02)--(a03);
\draw[epi] (a02)--(a12);
\draw[epi] (a03)--(a13);
\end{scope}
\end{scope}
\begin{scope}[xshift=2.7cm]
\draw[thick] (2.5,0.5) rectangle (4.5,-1.5);
\begin{scope}
\draw  (3, -1)   node(a12){{\footnotesize $A_{12}$}};
\draw  (4, 0)   node (a03){{\footnotesize $A_{03}$}};
\draw (4, -1)   node (a13){{\footnotesize $A_{13}$}};
\draw[mono] (a12)--(a13);
\draw[epi] (a03)--(a13);
\end{scope}
\end{scope}
\begin{scope}[xshift=-2.7cm]
\draw[thick] (2.5,0.5) rectangle (4.5,-1.5);
\begin{scope}
\draw  (3, -1)   node(a12){{\footnotesize $A_{12}$}};
\draw  (3, 0)   node(a02){{\footnotesize $A_{02}$}};
\draw  (4, 0)   node (a03){{\footnotesize $A_{03}$}};
\draw[mono] (a02)--(a03);
\draw[epi] (a02)--(a12);
\end{scope}
\end{scope}
\end{tikzpicture}
\end{equation}
\end{itemize}

Categorical structures that follow this scheme, roughly in the chronological order in which they were considered, include the case of $\cE$ being:
\begin{itemize}[leftmargin=*]
    \item an \emph{abelian category};
    a motivating example would be the abelian category of abelian groups, or more generally the abelian category of modules over a ring $R$.
    \item a \emph{(proto)exact category}, as in \cite[\textsection 1.2]{QuillenK} and \cite[\textsection 2.4]{DKbook}; a motivating example would be the exact category of finitely generated abelian groups, or more generally the exact category of finitely generated modules over a ring $R$, or the proto-exact category of finite pointed sets;
     \item a \emph{stable $\infty$-category}, as in \cite{LurieStable};
    a motivating example would be the stable $\infty$-category of spectra, or more generally the stable $\infty$-category of (compact) modules over a ring spectrum. 
    \item an \emph{exact $\infty$-category}, as in \cite[\textsection 7.2]{DKbook} (cf.~also \cite{BarwickExact}); a motivating example would be the exact $\infty$-category of connective spectra.
    \item a \emph{stable pointed double category} (or more generally a \emph{stable augmented double category}), as in \cite[\textsection3]{BOORS1}; a motivating example is the path construction of a (reduced) $2$-Segal set (and in fact it is shown as \cite[Theorem 6.1]{BOORS1} that all examples are of this kind).
    \item a \emph{stable pointed double Segal space}
    and more generally a \emph{stable augmented augmented double Segal space}, as in \cite[\textsection2.5]{BOORS3}; a motivating example is the path construction of a (reduced) $2$-Segal space (and in fact it is shown as \cite[Theorem~6.1]{BOORS3} that all examples are of this kind).
    \item a \emph{CGW category} as in \cite[\textsection2]{CZdevissage}, (cf.~also an \emph{ACGW category} as in \cite[\textsection5]{CZdevissage},  and an \emph{ECGW category} as in \cite[Part~1]{SarazolaShapiro}); a motivating example would be the category of reduced schemes of finite type.
    \item an \emph{isostable squares category} as in \cite[\textsection3]{CalleSarazola}; a motivating example is the squares category of polytopes from \cite[Examples~2.14]{CalleSarazola}
\end{itemize}

These different frameworks can be organized as follows:
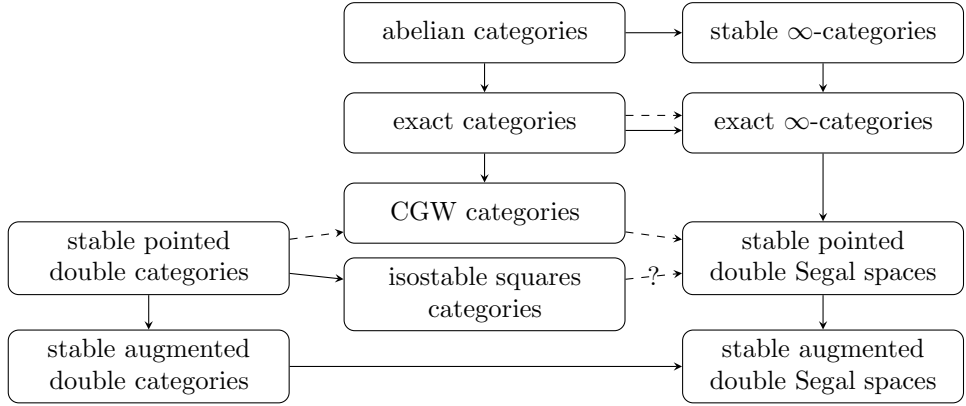
\begin{center}
 \begin{tikzpicture}
\def\vsp{1.2cm}
\def\hspo{3.5cm}
\def\hspt{4.5cm}
  \draw (-1+\hspo,-\vsp) node[category3] (abcat){abelian categories};
   \draw (-1+\hspo,-2*\vsp) node[category3] (exact){exact categories};
   \draw (-1+\hspo, -3*\vsp) node[category3] (CGW){CGW categories};
    \draw (-1+\hspo, -3.9*\vsp) node[category3] (isostable){isostable squares categories};
   \draw (-1,-3.5*\vsp) node[category3] (SPDC){stable pointed double categories};
    \draw (-1,-4.7*\vsp) node[category3] (SADC){stable augmented double categories};
  \draw (-1+\hspo+\hspt,-\vsp) node[category3] (stableIC){stable $\infty$-categories};
  \draw (-1+\hspo+\hspt, -2*\vsp) node[category3] (exactIC){exact $\infty$-categories};
  \draw (-1+\hspo+\hspt, -3.5*\vsp) node[category3] (SPDS){stable pointed double Segal spaces};
  \draw (-1+\hspo+\hspt, -4.7*\vsp) node[category3] (SADS){stable augmented double Segal spaces};
\draw[-stealth] (abcat)--(stableIC);
\draw[-stealth, dashed] ($(exact.east)+(0,0.1)$)--($(exactIC.west)+(0,0.1cm)$);
\draw[-stealth] ($(exact.east)+(0,-0.1)$)--($(exactIC.west)+(0,-0.1cm)$);
\draw[-stealth, dashed] (SPDC)--(CGW);
\draw[-stealth, dashed] (CGW)--(SPDS);
\draw[-stealth] (SPDC)--(isostable);
\draw[-stealth, dashed] (isostable)--node(label1){$?$}(SPDS);
\draw[-stealth] (SADC)--(SADS);
%
\draw[-stealth] (abcat)--(exact);
\draw[-stealth] (stableIC)--(exactIC);
\draw[-stealth] (exactIC)--(SPDS);
\draw[-stealth] (exact)--(CGW);
\draw[-stealth] (SPDC)--(SADC);
\draw[-stealth] (SPDS)--(SADS);
 \end{tikzpicture}
\end{center}
Here, the horizontal direction roughly encodes the generalization from the strict to a weaker metatheory, while the vertical direction progressively favors structure over properties.

Some existing comparisons between these frameworks include:
\cite[\textsection 2]{QuillenK} for how exact categories recover abelian categories,
\cite[Example 7.2.3]{DKbook} for how exact $\infty$-categories recover exact categories,
\cite{NeemanExact} for how exact $\infty$-categories possibly recover exact categories through the derived $\infty$-category of chain complexes,
\cite[\textsection 1.3.2]{LurieHA} for how stable $\infty$-categories recover abelian categories,
\cite[Example~3.1]{CZdevissage} for how CGW categories recover exact categories,
\cite[Proposition 3.11]{CalleSarazola} for how isostable squares categories recover stable pointed double categories,
\cite{BOORS4} for how stable augmented double Segal spaces recover exact categories and stable $\infty$-categories through appropriate nerve constructions.

The study of how CGW categories recover pointed stable double categories and of how pointed stable double Segal spaces recover CGW categories is subject of ongoing work by some workshop participants (cf.~\cite[pp.~29-30]{Oberwolfach2024})
A comment on how CGW categories possibly relate to isostable square categories is in \cite[Remark~3.27]{CalleSarazola}.

An appropriate $S_\bullet$-construction was considered for all these situations. This was done in \cite[\textsection1.3]{waldhausen} 
for exact categories (as a special case of Waldhausen categories, there referred to as categories with cofibrations and weak equivalences), in \cite[Recollection 5.8]{BarwickExact} (as a special case of Waldhausen $\infty$-categories)
and \cite[\textsection2.4]{DKbook} for exact $\infty$-categories, in \cite[\textsection3]{BOORS1} for pointed or augmented stable double categories, in \cite[\textsection2,~\textsection7]{BOORS3} for pointed or augmented stable double Segal spaces, and in \cite[\textsection2.2]{CalleSarazola} for isostable squares categories. Many are recovered as a special case of the $S_\bullet$-construction for $\Sigma$-spaces, which will be recalled in \cref{SdotForSigma}.

Results establishing the compatibility between the various $S_\bullet$-constructions include \cite[\textsection2,~\textsection3]{BOORS4} for the compatibility of the one for exact categories and stable augmented double categories with the one for stable augmented double Segal spaces, and \cite[Proposition~3.6]{CalleSarazola} for the compatibility of the one for stable augmented double categories with the one for isostable squares categories.

The versions of $S_\bullet$-constructions considered for different frameworks, however, don't always agree in a strict sense. For instance, the $S_\bullet$-construction for abelian categories and the one for stable $\infty$-categories only agree up to equivalence upon taking the loop space of its geometric realization, essentially as an instance of the Gillet--Waldhausen theorem (see~\cite[Theorem 1.11.7]{ThomasonTrobaugh}).

\subsubsection{Pointed semi-stable context}

\label{NonStable}

One could also generalize in a different direction, dropping one half of the stability axioms, asking for a square to be determined by the span contained in its boundary; that is, only the left half of \eqref{SpanCospan}. For instance, this is the case when the class of distinguished squares are certain pushout squares in a given category, which are not necessarily pullbacks.
Categorical structures fit this framework, roughly in the chronological order in which they were considered, include the case of $\cW$ being:
\begin{itemize}[leftmargin=*]
    \item a \emph{Waldhausen category} as in \cite[Definition~II.9.1.1]{WeibelKBook}
    (originally referred to as \emph{category with cofibrations and weak equivalences}
    in \cite[\textsection 1.2]{waldhausen});
    a motivating example would be the Waldhausen category of pointed CW-complexes, or more generally the category of retractive spaces over a CW-complex $X$ (see~\cite[\textsection 1.1]{waldhausen};
    \item a \emph{Waldhausen $\infty$-category}, as in \cite[Definition 2.7]{BarwickKtheory}; this notion offers a common framework to talk about stable $\infty$-categories and Waldhausen categories;
    \item a \emph{proto-Waldhausen squares category}, as in \cite[\textsection2.3]{CalleSarazola};
    \item a \emph{semi-stable pointed/augmented double category}, by which we mean a double category that satisfies only one half of the stability property from \cite[\textsection 3]{BOORS1}; that is, the one corresponding to the left side of \eqref{SpanCospan}.
     \item a \emph{semi-stable pointed/augmented double Segal space}, by which we mean a double Segal space that only one half of the stability property from \cite[\textsection 2.5]{BOORS3}; that is, the one corresponding to the left side of \eqref{SpanCospan}.
\end{itemize}
These different frameworks could be organized as follows:
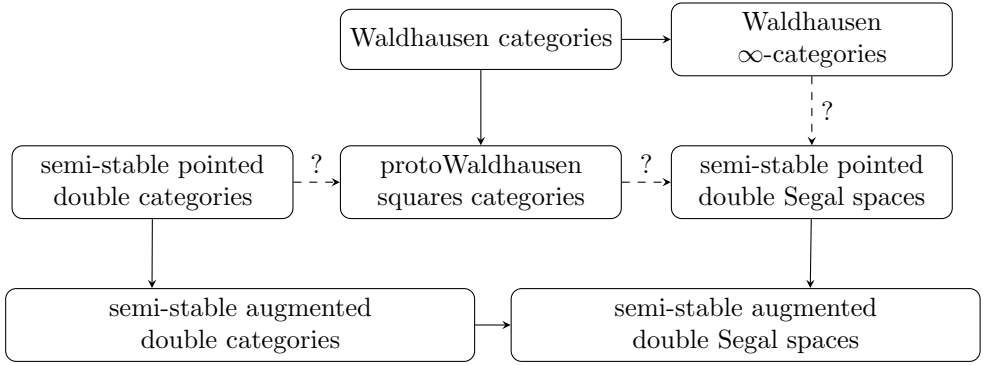
\begin{center}
 \begin{tikzpicture}
 \def\vsp{1.9cm}
 \def\hspo{3.4cm}
  \def\hspt{4.4cm}
   \draw (-1+\hspo,-2*\vsp) node[category4] (Wald){Waldhausen categories};
   \draw (-1+\hspo, -3*\vsp) node[category4] (protoWald){protoWaldhausen squares categories};
   \draw (-1,-3*\vsp) node[category4] (SPDC){semi-stable pointed double categories};
    \draw (-1+0.2cm,-4*\vsp) node[category7] (SADC){semi-stable augmented double categories};
   %
  \draw (-1+\hspo+\hspt, -2*\vsp) node[category4] (WaldIC){Waldhausen $\infty$-categories};
  \draw (-1+\hspo+\hspt, -3*\vsp) node[category4] (SPDS){semi-stable pointed double Segal spaces};
 \draw (-1+\hspo+0.8*\hspt, -4*\vsp) node[category7] (SADS){semi-stable augmented double Segal spaces}; 
\draw[-stealth] (Wald)--(WaldIC);
\draw[-stealth, dashed] (SPDC)--node[above](label1){$?$}(protoWald);
\draw[-stealth, dashed] (protoWald)--node[above](label2){$?$}(SPDS);
\draw[-stealth] (SADC)--(SADS);
%
\draw[-stealth] (Wald)--(protoWald);
\draw[-stealth, dashed] (WaldIC)--node[right](label3){$?$}(SPDS);
\draw[-stealth] (SPDC)--($(SADC.north west)!0.625!(SADC.north)$);
\draw[-stealth] (SPDS)--($(SADS.north east)!0.725!(SADS.north)$);
 \end{tikzpicture}
\end{center}

An appropriate $S_\bullet$-construction was considered for several of these situations. This was done in \cite[\textsection 1.3]{waldhausen}
for $\cW$ being a Waldhausen category, in \cite[\textsection 5]{BarwickKtheory} (through a slightly different though equivalent formalism)
for $\cW$ being a Waldhausen $\infty$-category, and as \cite[\textsection2.2]{CalleSarazola} for $\cW$ being a proto-Waldhausen squares category. An $S_\bullet$-construction for stable augmented double Segal spaces (in fact for any $\Sigma$-space) is given in \cref{SdotForSigma}.

Sometimes, however, the versions of $S_\bullet$-construction considered for different framework don't agree in a strict sense. For instance, as explained in \cite[Remark 2.34]{CalleSarazola}, the $S_\bullet$-construction for proto-Waldhausen categories and the one for Waldhausen categories only agree up to equivalence upon taking geometric realization.

\subsubsection{General context}

\label{SdotForSigma}
We describe here a
general framework for $S_\bullet$-const\-ruc\-tion recovering many of the discussed instances.

Let $\Sigma$ denote the category from \cite[Definition~2.9]{BOORS3}, obtained by freely adding a terminal object $[-1]$ to the category $\Delta\times\Delta$, and consider the category $\cS^{\Sigma^{\op}}$ of $\Sigma$-spaces. Roughly speaking, a $\Sigma$-space $X$ consists of a bisimplicial space $X_{\bullet,\bullet}$ together with an augmentation space $X_{-1}$ and an augmentation map $X_{-1}\to X_{0,0}$.

We briefly recall how the framework of $\Sigma$-spaces recovers the framework of exact categories through an appropriate nerve construction, as done in \cite[\textsection 2]{BOORS4}. Studied nerve constructions for other frameworks from \cref{stable,NonStable} include
\cite[\textsection 3]{BOORS4} for stable $\infty$-categories and \cite[\textsection 4]{BOORS4} for exact $\infty$-categories.

The datum of an \emph{exact category} $\cE$ consists of a category endowed with two classes of distinguished morphisms, called \emph{admissible monomorphisms} and \emph{admissible epimorphisms}, and a distinguished class of objects.  We refer the reader to \cite[\textsection 2]{QuillenK}
for a complete definition, and we just mention that the axioms imposed on $\cE$ guarantee that \emph{zero objects} exist, that admissible monomorphisms (resp.~admissible epimorphisms) can be composed, that pullback (resp.~pushout) of an admissible monomorphism (resp.~an admissible epimorphism) exists and is still an admissible monomorphism (resp.~an admissible epimorphism). This is in particular enough structure to talk about admissible \emph{short exact sequences} and \emph{bicartesian squares}. As a consequence of the axioms, appropriate versions of the properties of pointedness and stability follow.
An \emph{exact functor} is a functor that preserves the admissible short exact sequences in an appropriate way.

In most examples of exact categories $\cE_R$ based on categories of modules over a ring $R$, the admissible monomorphisms (resp.~admissible epimorphisms) are the module maps that are injective (resp.~surjective), and the zero objects are the trivial modules.

If $\cS$ denotes the category of spaces and $\cE x\cat$ denotes the category of exact categories and exact functors,
we recall the exact nerve construction $N^{\mathrm{ex}}\colon\cE x\cat\to\cS^{\Sigma^{\op}}$ from \cite[Definition~2.2]{BOORS4}.
Given an exact category $\cE$, let 
\[
N_{-1}^{\mathrm{ex}}(\cE)\coloneqq\mathrm Z(\cE)
\]
denote the groupoid of zero objects in $\cE$, and for $a,b\geq0$ we let
\[N^{\mathrm{ex}}_{a,b}(\cE)\coloneqq\Hom^{\mathrm{ex}}([a]\times[b],\cE)\]
denote the groupoid of \emph{exact functors} from $[a]\times[b]$ to $\cE$; that is, of functors $[a]\times[b]\to\cE_R$ which:
\begin{itemize}[leftmargin=*]
\item send the morphisms of the form $(i,j)\to(i,j+\ell)$ to  monomorphisms in $\cE_R$,
\item send the morphisms $(i,j)\to(i+k,j)$ to epimorphisms in $\cE_R$,
\item sends the commutative squares involving two maps of the form $(i,j)\to (i+k,j)$  and two maps of the form $(i,j)\to (i,j+\ell)$ to bicartesian squares in $\cE_R$.
\end{itemize}
For instance, a functor $[2]\times[3]\to\cE_R$
defining an object of $N^{\mathrm{ex}}_{2,3}(\cE)$ is of the form
\begin{center}
\begin{tikzpicture}[scale=0.9, inner sep=0pt, font=\footnotesize, baseline=-30pt]
\begin{scope}[xshift=0.0cm]
\draw[thick] (0.5,0.5) rectangle (4.5,-2.5);
\begin{scope}[xshift=0.0cm]
   \draw (1,0) node(a00){{\footnotesize $A_{00}$}};
      \draw (1,-1) node(a10){{\footnotesize $A_{10}$}};
         \draw (1,-2) node(a20){{\footnotesize $A_{20}$}};
\draw (2,0) node(a01){{\footnotesize $A_{01}$}};
\draw (2,-1) node(a11) {{\footnotesize $A_{11}$}};
\draw  (3, -1)   node(a12){{\footnotesize $A_{12}$}};
\draw  (3, 0)   node(a02){{\footnotesize $A_{02}$}};
\draw  (4, 0)   node (a03){{\footnotesize $A_{03}$}};
\draw (4, -1)   node (a13){{\footnotesize $A_{13}$}};
\draw (4, -2)   node (a23){{\footnotesize $A_{23}$}};
\draw (2,-2) node (a21) {{\footnotesize $A_{21}$}};
\draw (3,-2) node (a22) {{\footnotesize $A_{22}$}};
\draw[mono] (a11)--(a12);
\draw[mono] (a12)--(a13);
\draw[mono] (a00)--(a01);
\draw[mono] (a10)--(a11);
\draw[mono] (a20)--(a21);
\draw[mono] (a11)--(a12);
\draw[mono] (a01)--(a02);
\draw[mono] (a02)--(a03);
\draw[mono] (a21)--(a22);
\draw[mono] (a22)--(a23);
\draw[epi] (a00)--(a10);
\draw[epi] (a10)--(a20);
\draw[epi] (a01)--(a11);
\draw[epi] (a11)--(a21);
\draw[epi] (a12)--(a22);
\draw[epi] (a02)--(a12);
\draw[epi] (a03)--(a13);
\draw[epi] (a13)--(a23);
\end{scope}
\end{scope}
%
\end{tikzpicture}
\end{center}
with the convention that all horizontal morphisms are monomorphisms, all vertical morphisms are epimorphisms, and all squares are bicartesian squares.

We now turn to the $S_\bullet$-construction in the context of $\Sigma$-spaces. Let $p\colon\Sigma\to\Delta$ denote the ordinal sum given by $[-1]\mapsto[0]$ and $[a,b]\mapsto[a+1+b]$, as considered in \cite[Definition 2.15]{BOORS3} (see \cite{RovelliChapter} in this volume). If  $\cS$ denotes the category of spaces, there is an induced adjunction
\[\cP=p^*\colon\cS^{\Delta^{\op}}\leftrightarrows\cS^{\Sigma^{\op}}\colon p_*=S_{\bullet}\]
between the category $\cS^{\Delta^{\op}}$ of simplicial spaces and the category $\cS^{\Sigma^{\op}}$ of $\Sigma$-spaces.

So we obtain the description of the $S_\bullet$-construction:

\begin{const}
For $k\geq0$, and $X$ a $\Sigma$-space, we let
\[S_k(X)\coloneqq\Map(\cP\Delta[k],X)\]
denote the space of maps of $\Sigma$-spaces $\cP\Delta[k]\to X$. For $0\leq i\leq n$, the $i$-th \emph{face map} $d_i\colon S_k(X)\to S_{k-1}(X)$ and the $i$-th \emph{degeneracy map} $s_i\colon S_k(X)\to S_{k+1}(X)$ are induced by the cosimplicial structure of $\cP\Delta[\bullet]$, and can be described in a similar fashion to those of \cref{Sdisc}.
\end{const}

In total we get the \emph{$S_\bullet$-construction} of any $\Sigma$-space $X$:

\begin{prop}
If $X$ is a $\Sigma$-space, the assignment $[k]\mapsto S_k(X) $ defines a simplicial space $S_\bullet(X)$.
\end{prop}

The $S_\bullet$-construction for exact categories and the one for $\Sigma$-spaces are compatible through the exact nerve $N^{\mathrm{ex}}\colon\cE x\cat\to\cS^{\Sigma^{\op}}$:

\begin{rmk}
It is shown as \cite[Theorem 2.18]{BOORS4} that there is a commutative diagram of categories
\[
\begin{tikzcd}
\cE x\cat\arrow[d,"N^{\mathrm{ex}}" swap]\arrow[r,"S_{\bullet}"]&\gpd^{\Delta^{\op}}\arrow[d,"N_*"]\\
\cS^{\Sigma^{\op}}\arrow[r,"S_{\bullet}" swap]&\cS^{\Delta^{\op}}
\end{tikzcd}
\]
Indeed, for $n\geq0$ and $\cE$ an exact category, there is a natural bijection
\[
S_{k}(N^{\mathrm{ex}}\cE)\cong  \Map(\cP\Delta[k],N^{\mathrm{ex}}\cE)
 \cong N(\Hom^{\mathrm{ex}}(\Ar[k],\cE))\cong N(S_{k}(\cE))=(N_*S)_{k}(\cE).
\]
\end{rmk}

Through this generalized version of the $S_\bullet$-construction, one can now define the $K$-theory space for a very general type of input.
 If $X$ is a $\Sigma$-space
\[
K(X) \coloneqq \Omega(\mathrm{Real}S_\bullet(X)).\]

\subsection{$2$-Segality properties}

Recall the definition of \emph{$2$-Segal space} from
\linebreak
\cite{DKbook} (see \cite{SternChapter} in this volume) -- a.k.a.~\emph{decomposition space in \cite{GCKT1}} (see \cite{HackneyChapter} in this volume) -- and its generalization \emph{lower $2$-Segal space} from \cite[Definition~2.2]{Poguntke}.

We discuss how the properties of the input $\cE$ reflect into properties its $2$-Segality $S_\bullet$-construction. We will compose any existing $S_\bullet$-construction valued in $\gpd^{\Delta^{\op}}$ from previous sections with the levelwise nerve functor $N_*\colon\gpd^{\Delta^{\op}}\to\cS^{\Delta^{\op}}$, keeping the same notation. The overall expectation is the following:
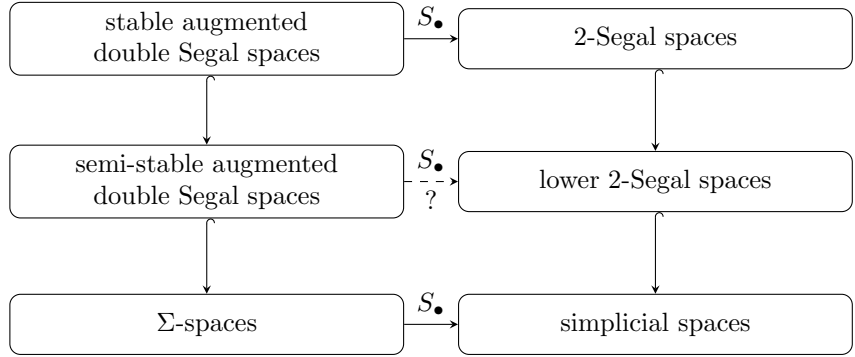
\begin{center}
 \begin{tikzpicture}
 \def\vsp{1.9cm}
 \def\hspo{5cm}
  \def\hspt{4.4cm}
   \draw (-1+\hspo,-2*\vsp) node[category5] (2Seg){2-Segal spaces};
   \draw (-1+\hspo, -3*\vsp) node[category5] (2SegPlus){lower 2-Segal spaces};
    \draw (-1+\hspo, -4*\vsp) node[category5] (SimpSp){simplicial spaces};
   \draw (-1,-2*\vsp) node[category5] (sadss){stable augmented double Segal spaces};
    \draw (-1,-3*\vsp) node[category5] (hsadss){semi-stable augmented double Segal spaces};
      \draw (-1,-4*\vsp) node[category5] (SigmaSp){$\Sigma$-spaces};
   %
\draw[-stealth] (sadss)--node[above]{$S_\bullet$}(2Seg);
\draw[-stealth, dashed] (hsadss)--node[above](abovelabel){$S_\bullet$} node[below](label2){$?$}(2SegPlus);
\draw[-stealth] (SigmaSp)--node[above]{$S_\bullet$}(SimpSp);
\draw[right hook-stealth] (sadss)--(hsadss);
\draw[right hook-stealth] (hsadss)--(SigmaSp);
\draw[right hook-stealth] (2Seg)--(2SegPlus);
 \draw[right hook-stealth] (2SegPlus)--(SimpSp);
 \end{tikzpicture}
\end{center}

\subsubsection{The $2$-Segal case}

\label{2Segal}

We start by treating the stable pointed case, namely the case of $\cE$ being one of the structures from \cref{stable}.

\begin{thm}
\label{2SegalityThm}
   If $\cE$ is any of the structures discussed in \cref{stable},
   the $S_\bullet$-construction $S_\bullet\cE$
    is a $2$-Segal space.
\end{thm}

This is shown
as \cite[Theorem 10.10]{GCKT1} for $\cE$ being an abelian category,
as \cite[\textsection 2.4]{DKbook} for $\cE$ being a (proto-)exact category, as \cite[Theorem 4.8]{BOORS1} for $\cE$ being a stable augmented double category, 
as \cite[Theorem~5.1]{BOORS3} for $\cE$ being a stable augmented double Segal space,
as \cite[Theorem~B]{CalleSarazola} for $\cE$ being an isostable
squares category. We briefly illustrate how to treat one of the lowest dimensional $2$-Segal maps:
\begin{proof}[Idea]
    The $2$-Segal map
    \[(d_2,d_0)\colon S_3\cE\to S_2\cE\times_{S_1\cE}S_2\cE\]
   can be depicted as
       
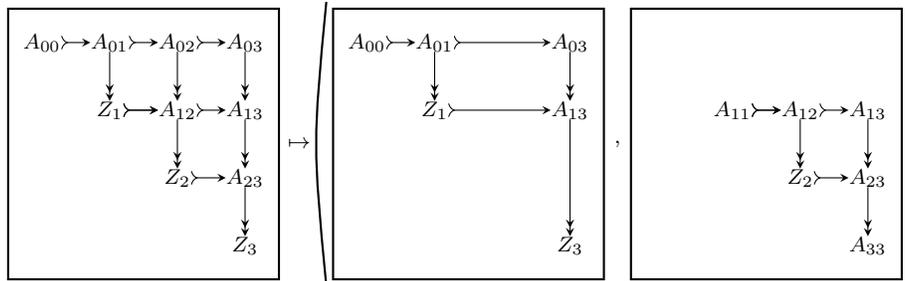
\begin{center}
\begin{tikzpicture}[scale=0.9, inner sep=0pt, font=\footnotesize, baseline=-30pt]
\begin{scope}[xshift=0.0cm]
\draw[thick] (0.5,0.5) rectangle (4.5,-3.5);
\begin{scope}
    \draw (1,0) node(a00){{\footnotesize $A_{00}$}};
\draw (2,0) node(a01){{\footnotesize $A_{01}$}};
\draw (2,-1) node(a11) {{\footnotesize $Z_{1}$}};
\draw  (3, -1)   node(a12){{\footnotesize $A_{12}$}};
\draw  (3, 0)   node(a02){{\footnotesize $A_{02}$}};
\draw (3,-2) node (a22){{\footnotesize $Z_{2}$}};
\draw  (4, 0)   node (a03){{\footnotesize $A_{03}$}};
\draw (4, -1)   node (a13){{\footnotesize $A_{13}$}};
\draw (4, -2)   node (a23){{\footnotesize $A_{23}$}};
\draw (4, -3) node (a33){{\footnotesize $Z_{3}$}};
\draw[mono] (a11)--(a12);
\draw[mono] (a12)--(a13);
\draw[mono] (a22)--(a23);
\draw[mono] (a00)--(a01);
\draw[mono] (a11)--(a12);
\draw[mono] (a01)--(a02);
\draw[mono] (a02)--(a03);
\draw[epi] (a01)--(a11);
\draw[epi] (a12)--(a22);
\draw[epi] (a02)--(a12);
\draw[epi] (a03)--(a13);
\draw[epi] (a13)--(a23);
\draw[epi] (a23)--(a33);
\end{scope}
\end{scope}
\draw (4.8,-1.5) node (si){$\mapsto$};
\draw[thick] (5.2,0.6)..controls (5,-1.5)..(5.2,-3.6);
\begin{scope}[xshift=9.2cm]
\draw[thick] (0.5,0.5) rectangle (4.5,-3.5);
\begin{scope}
\draw (2,-1) node(a11) {{\footnotesize $A_{11}$}};
\draw  (3, -1)   node(a12){{\footnotesize $A_{12}$}};
\draw (3,-2) node (a22){{\footnotesize $Z_{2}$}};
\draw (4, -1)   node (a13){{\footnotesize $A_{13}$}};
\draw (4, -2)   node (a23){{\footnotesize $A_{23}$}};
\draw (4, -3) node (a33){{\footnotesize $A_{33}$}};
\draw[mono] (a11)--(a12);
\draw[mono] (a12)--(a13);
\draw[mono] (a22)--(a23);
\draw[mono] (a11)--(a12);
%
\draw[epi] (a12)--(a22);
\draw[epi] (a13)--(a23);
\draw[epi] (a23)--(a33);
\end{scope}
\end{scope}

\draw (9.5,-1.5) node (si){$,$};

\begin{scope}[xshift=4.8cm]
\draw[thick] (0.5,0.5) rectangle (4.5,-3.5);
\begin{scope}
    \draw (1,0) node(a00){{\footnotesize $A_{00}$}};
\draw (2,0) node(a01){{\footnotesize $A_{01}$}};
\draw (2,-1) node(a11) {{\footnotesize $Z_{1}$}};
\draw  (4, 0)   node (a03){{\footnotesize $A_{03}$}};
\draw (4, -1)   node (a13){{\footnotesize $A_{13}$}};
\draw (4, -3) node (a33){{\footnotesize $Z_{3}$}};
\draw[mono] (a11)--(a13);
\draw[mono] (a00)--(a01);
\draw[mono] (a01)--(a03);
%
\draw[epi] (a01)--(a11);
\draw[epi] (a03)--(a13);
\draw[epi] (a13)--(a33);
\end{scope}
\end{scope}

\draw[thick] (13.8,0.6)..controls (14,-1.5)..(13.8,-3.6);
\end{tikzpicture}
\end{center}
Using the structure and properties of $\cE$, one can consider the assignment
 
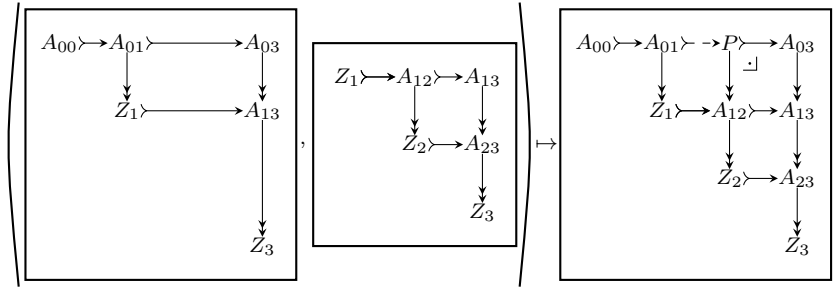
\begin{center}
\begin{tikzpicture}[scale=0.9, inner sep=0pt, font=\footnotesize, baseline=-30pt]
\def\l{1cm}
\begin{scope}[xshift={9.1cm-1.2*\l}]
\draw[thick] (0.5,0.5) rectangle (4.5,-3.5);
\begin{scope}
    \draw (1,0) node(a00){{\footnotesize $A_{00}$}};
\draw (2,0) node(a01){{\footnotesize $A_{01}$}};
\draw (2,-1) node(a11) {{\footnotesize $Z_{1}$}};
\draw  (3, -1)   node(a12){{\footnotesize $A_{12}$}};
\draw  (3, 0)   node(a02){{\footnotesize $P$}};
\draw (3,-2) node (a22){{\footnotesize $Z_{2}$}};
\draw  (4, 0)   node (a03){{\footnotesize $A_{03}$}};
\draw (4, -1)   node (a13){{\footnotesize $A_{13}$}};
\draw (4, -2)   node (a23){{\footnotesize $A_{23}$}};
\draw (4, -3) node (a33){{\footnotesize $Z_{3}$}};
\draw[mono] (a11)--(a12);
\draw[mono] (a12)--(a13);
\draw[mono] (a22)--(a23);
\draw[mono] (a00)--(a01);
\draw[mono] (a11)--(a12);
\draw[mono, dashed] (a01)--(a02);
\draw[mono] (a02)--(a03);
\draw[epi] (a01)--(a11);
\draw[epi] (a12)--(a22);
\draw[epi] (a02)--(a12);
\draw[epi] (a03)--(a13);
\draw[epi] (a13)--(a23);
\draw[epi] (a23)--(a33);
\begin{scope}[shift=($(a02)!.4!(a13)$), scale=0.4]
\draw +(-.5,0) -- +(0,0)  -- +(0,.5);
\fill +(-.25,.25) circle (.05);
\end{scope}
\end{scope}
\end{scope}
\draw ({9.5cm-1.3*\l},-1.5) node (si){$\mapsto$};
\draw[thick] (0.4,0.6)..controls (0.2,-1.5)..(0.4,-3.6);
\begin{scope}[xshift=3.25*\l, yshift=0.5*\l]
\draw[thick] (1.5,-0.5) rectangle (4.5,-3.5);
\begin{scope}
\draw (2,-1) node(a11) {{\footnotesize $Z_{1}$}};
\draw  (3, -1)   node(a12){{\footnotesize $A_{12}$}};
\draw (3,-2) node (a22){{\footnotesize $Z_{2}$}};
\draw (4, -1)   node (a13){{\footnotesize $A_{13}$}};
\draw (4, -2)   node (a23){{\footnotesize $A_{23}$}};
\draw (4, -3) node (a33){{\footnotesize $Z_{3}$}};
\draw[mono] (a11)--(a12);
\draw[mono] (a12)--(a13);
\draw[mono] (a22)--(a23);
\draw[mono] (a11)--(a12);
%
\draw[epi] (a12)--(a22);
\draw[epi] (a13)--(a23);
\draw[epi] (a23)--(a33);
\end{scope}
\end{scope}

\draw ({4.7cm-0.1*\l},-1.5) node (comma){$,$};

\begin{scope}
\draw[thick] (0.5,0.5) rectangle (4.5,-3.5);
\begin{scope}
    \draw (1,0) node(a00){{\footnotesize $A_{00}$}};
\draw (2,0) node(a01){{\footnotesize $A_{01}$}};
\draw (2,-1) node(a11) {{\footnotesize $Z_{1}$}};
\draw  (4, 0)   node (a03){{\footnotesize $A_{03}$}};
\draw (4, -1)   node (a13){{\footnotesize $A_{13}$}};
\draw (4, -3) node (a33){{\footnotesize $Z_{3}$}};
\draw[mono] (a11)--(a13);
\draw[mono] (a00)--(a01);
\draw[mono] (a01)--(a03);
%
\draw[epi] (a01)--(a11);
\draw[epi] (a03)--(a13);
\draw[epi] (a13)--(a33);
\end{scope}
\end{scope}

\draw[thick] ({9.1cm-1.3*\l},0.6)..controls ({9.3cm-1.3*\l},-1.5)..({9.1cm-1.3*\l},-3.6);

\end{tikzpicture}
\end{center}
and verify that it defines an inverse equivalence of groupoids
    \[ S_2\cE\times_{S_1\cE}S_2\cE\xrightarrow{\simeq} S_3\cE.\]
    The other map
    \[(d_1,d_3)\colon S_3(\cE)\to S_2(\cE)\times_{S_1(\cE)}S_2(\cE)\]
    can be treated analogously.
\end{proof}

\begin{rmk}
    One could also see that $S_\bullet(\cE)$ is a $2$-Segal space by observing that its edgewise subdivision $\mathrm{esd}(S_\bullet(\cE))$ is a Segal space (which is closely related to Quillen's $Q$-construction $Q(\cE)$ from \cite[\textsection I.2]{QuillenK}), hence by \cite[Theorem 2.11]{BOORS2} the $S_\bullet$-construction $S_\bullet(\cE)$ is a $2$-Segal space.
\end{rmk}

It was shown in \cite{BOORS3} (and \cite{BOORS1} for the discrete version of the statement, see also \cite{RovelliChapter} in this volume) that in fact all $2$-Segal spaces can be obtained through the $S_\bullet$-construction:

\begin{thm}
The $S_\bullet$-construction is an equivalence between $2$-Segal spaces and stable augmented double Segal spaces.
\end{thm}

\subsubsection{The lower $2$-Segal case}

\label{lower2Segal}

We now treat the pointed semi-stable cases, namely the case of $\cW$ being one of those from \cref{NonStable}. While one half of the proof of \cref{2SegalityThm} (and of the argument we sketched in \cref{NonStable}) can be reproduced at this further level of generality, the other half cannot.

In fact, lacking one half of the stability conditions, the $S_\bullet$-construction from \cref{NonStable} does not generally define a $2$-Segal space. To see this, let's consider the following explicit example (obtained by adjusting the one from \cite[Example~4.3]{BOORS2}.

\begin{rmk}
\label{Snot2Segal}
Let $\cW$ denote the Waldhausen category of finite pointed CW-complexes and pointed
cellular maps, with cellular embeddings as cofibrations and homotopy equivalences
as weak equivalences (a special case of \cite[\textsection1.1-1.2]{waldhausen}
with $X$ being a point).
Then we can see that the $2$-Segal map
    \[(d_2,d_0)\colon S_3(\cW)\to S_2(\cW)\times_{S_1(\cW)}S_2(\cW)\]
    is not injective on connected components, hence in particular it is not an equivalence.
    Indeed, if $P$ is a finite $2$-dimensional CW-complex which is not contractible but whose
suspension is contractible (for example, the classifying space of the perfect group
from \cite[Example 2.38]{HatcherAT}), the following two objects of $S_3(\cW)$ live in different connected components but map to the same connected component
in $S_2(\cW)\times_{S_1(\cW)}S_2(\cW)$ under the map $(d_2,d_0)$:
        
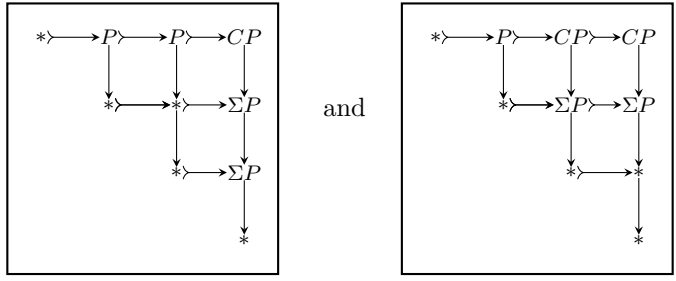
\begin{center}
\begin{tikzpicture}[scale=0.9, inner sep=0pt, font=\footnotesize, baseline=-30pt]
\begin{scope}[xshift=0.0cm]
\draw[thick] (0.5,0.5) rectangle (4.5,-3.5);
\begin{scope}
    \draw (1,0) node(a00){{\footnotesize $*$}};
\draw (2,0) node(a01){{\footnotesize $P$}};
\draw (2,-1) node(a11) {{\footnotesize $*$}};
\draw  (3, -1)   node(a12){{\footnotesize $*$}};
\draw  (3, 0)   node(a02){{\footnotesize $P$}};
\draw (3,-2) node (a22){{\footnotesize $*$}};
\draw  (4, 0)   node (a03){{\footnotesize $C P$}};
\draw (4, -1)   node (a13){{\footnotesize $\Sigma P$}};
\draw (4, -2)   node (a23){{\footnotesize $\Sigma P$}};
\draw (4, -3) node (a33){{\footnotesize $*$}};
\draw[mono] (a11)--(a12);
\draw[mono] (a12)--(a13);
\draw[mono] (a22)--(a23);
\draw[mono] (a00)--(a01);
\draw[mono] (a11)--(a12);
\draw[mono] (a01)--(a02);
\draw[mono] (a02)--(a03);
\draw[-stealth] (a01)--(a11);
\draw[-stealth] (a12)--(a22);
\draw[-stealth] (a02)--(a12);
\draw[-stealth] (a03)--(a13);
\draw[-stealth] (a13)--(a23);
\draw[-stealth] (a23)--(a33);
\end{scope}
\end{scope}
\end{tikzpicture}
\quad\text{ and }\quad
\begin{tikzpicture}[scale=0.9, inner sep=0pt, font=\footnotesize, baseline=-30pt]
\begin{scope}[xshift=0.0cm]
\draw[thick] (0.5,0.5) rectangle (4.5,-3.5);
\begin{scope}
    \draw (1,0) node(a00){{\footnotesize $*$}};
\draw (2,0) node(a01){{\footnotesize $P$}};
\draw (2,-1) node(a11) {{\footnotesize $*$}};
\draw  (3, -1)   node(a12){{\footnotesize $\Sigma P$}};
\draw  (3, 0)   node(a02){{\footnotesize $CP$}};
\draw (3,-2) node (a22){{\footnotesize $*$}};
\draw  (4, 0)   node (a03){{\footnotesize $C P$}};
\draw (4, -1)   node (a13){{\footnotesize $\Sigma P$}};
\draw (4, -2)   node (a23){{\footnotesize $*$}};
\draw (4, -3) node (a33){{\footnotesize $*$}};
\draw[mono] (a11)--(a12);
\draw[mono] (a12)--(a13);
\draw[mono] (a22)--(a23);
\draw[mono] (a00)--(a01);
\draw[mono] (a11)--(a12);
\draw[mono] (a01)--(a02);
\draw[mono] (a02)--(a03);
\draw[-stealth] (a01)--(a11);
\draw[-stealth] (a12)--(a22);
\draw[-stealth] (a02)--(a12);
\draw[-stealth] (a03)--(a13);
\draw[-stealth] (a13)--(a23);
\draw[-stealth] (a23)--(a33);
\end{scope}
\end{scope}
\end{tikzpicture}
\end{center}
    In particular, the simplicial space $S_\bullet(\cW)$ is \emph{not} a $2$-Segal space.
\end{rmk}

Others are expected to suffer from the same flaw. However, one half of the proof of \cref{2SegalityThm} should survive:

\begin{expectation}
\label{Lower2Segal}
If $\cW$ is any of the structures discussed in \cref{NonStable} the $S_\bullet$-construction $S_\bullet\cW$
    is a lower $2$-Segal space.
\end{expectation}

This was shown in \cite[\textsection7,~\textsection8]{CarawanWaldhausen}
for $\cW$ being a Waldhausen category, and should be true for all listed examples.

One could study whether all lower $2$-Segal spaces arise as $S_\bullet$-construction:

\begin{question}
Is this $S_\bullet$-construction an equivalence between lower $2$-Segal spaces and semi-stable augmented
double Segal spaces? If not, what is its homotopy essential image?
\end{question}
 
This question is possibly treatable by adjusting the argument from \cite{BOORS3} (or \cite{BOORS1} for the discrete case).
              
\section{Iterated Waldhausen's $S_\bullet$-construction and $K$-theory spectrum}

\label{Section2}

We discuss how to deloop $K$-theory through an appropriate iteration of Waldhausen's construction under appropriate assumptions.

\subsection{Waldhausen's iterated $S_\bullet$-construction}

We recall Waldhausen's technique to deloop the $K$-theory space of $\cE_R$.

\begin{rmk}
If $\cE$ is an exact category, there is an exact category $[\Ar[k],\cE]$.
The objects are the exact functors from $\Ar[k]$ to $\cE$, and the exact structure is inherited from that of $\cE$ (and is essentially a special case of \cite[Lemma 3.15]{Penney}).
More precisely, the ordinary $S_\bullet$-construction can be lifted to a functor
\[
\widetilde S_\bullet\coloneqq[\Ar[\bullet],-]\colon\cE x\cat\to\cE x\cat^{\Delta^{\op}}.
\]
\end{rmk}

Using this, and inspired by \cite[(3.6)]{Penney} for the case $n=2$, we can define the \emph{iterated} $S_\bullet$-construction -- the multisimplicial groupoid $S^{(n)}_{\bullet,\dots,\bullet}(\cE_R)$ -- as follows:

\begin{const}
For $n\geq0$ and $ k_1,\dots, k_n\geq0$, let $S^{(n)}_{ k_1,\dots, k_n}(\cE_R)$ denote the groupoid
\[S^{(n)}_{ k_1,\dots, k_n}(\cE_R)\coloneqq\Hom^{\mathrm{ex}}(\Ar[ k_1]\times\dots\times\Ar[ k_n],\cE_R)\]
of \emph{exact functors} 
from $\Ar[ k_1]\times\dots\times\Ar[ k_n]$ to $\cE_R$. That is, of functors from $\Ar[k_1]\times\dots\times\Ar[k_n]$ to $\cE_R$ which:
\begin{itemize}[leftmargin=*]
\item send the objects of the form $((i_1,i_1),\dots,(i_n,i_n))$ to a zero object in $\cE_R$,
\item send the morphisms of the form \[((i_1,j_1),\dots,(i_n,j_n))\to((i_1,j_1+\ell_1),\dots,(i_n,j_n+\ell_n))\]
to monomorphisms in $\cE_R$,
\item send the morphisms of the form \[((i_1,j_1),\dots,(i_n,j_n))\to((i_1+k_1,j_1),\dots,(i_n+k_n,j_n))\]
to epimorphisms in $\cE_R$,
\item sends the commutative squares involving two maps of the form \[((i_1,j_1),\dots,(i_n,j_n))\to((i_1+k_1,j_1),\dots,(i_n+k_n,j_n))\]  and two maps of the form \[((i_1,j_1),\dots,(i_n,j_n))\to((i_1,j_1+\ell_1),\dots,(i_n,j_n+\ell_n))\] to bicartesian squares in $\cE_R$.
\end{itemize}
\end{const}

So in total we get a multisimplicial groupoid:

\begin{prop}
The assignment $[ k_1,\dots, k_n]\mapsto S^{(n)}_{ k_1,\dots, k_n}(\cE_R) $ defines an $n$-fold simplicial groupoid $S^{(n)}_{\bullet,\dots,\bullet}(\cE_R) $.
\end{prop}

One can understand $S^{(n)}_{\bullet,\dots,\bullet}$ 
as an analog $n$-fold iteration of $S_\bullet$, which justifies the name ``iterated'' $S_{\bullet}$-construction (reconciling with the original viewpoint from \cite[\textsection 1.3]{waldhausen},
\cite[\textsection IV.8.5.5]{WeibelKBook} and \cite[Corollary 3.18]{Penney} for $n=2$). Let's explain this more in detail for the case $n=2$.

\begin{rmk}
\label{iteration}
If $\Ob\colon\cE x\cat\to\gpd$ denotes the core functor, there is a commutative diagram of categories:
\[
\begin{tikzcd}
\cE x\cat\arrow[r,"\widetilde{S}_\bullet"]\arrow[drr,"S^{(2)}_{\bullet,\bullet}" swap]&\cE x\cat^{\Delta^{\op}}\arrow[r,"(\widetilde{S}_\bullet)_*"]&(\cE x\cat^{\Delta^{\op}})^{\Delta^{\op}}\arrow[d,"(((\Ob)_*)_*"]\\
&&((\gpd^{\Delta^{\op}})^{\Delta^{\op}}
\end{tikzcd}
\]
Indeed, given an exact category $\cE$ and $a,b\geq0$, there are natural bijections for $a,b\geq0$:
\[
\begin{array}{lll}
\Ob ((\widetilde{S}_a)_*(\widetilde{S}_b(\cE)))& \cong &\Ob ((\widetilde{S}_a)_*([\Ar[b],\cE]))\\
     &\cong&\Ob([\Ar[a],[\Ar[b],\cE]])\\
     &\cong&\Hom^{\mathrm{ex}}(\Ar[a],[\Ar[b],\cE])\\
     &\cong&\Hom^{\mathrm{ex}}(\Ar[a]\times \Ar[b],\cE)= S_{a,b}^{(2)}(\cE).
\end{array}\]
\end{rmk}

Denote by $\mathrm{Real}^{(n)}\colon\gpd^{\Delta^{\op}\times\dots\times \Delta^{\op}}\to\cS^{\Delta^{\op}\times\dots\times \Delta^{\op}}\to\cS$ the realization of $n$-fold simplicial groupoids through $n$-fold simplicial spaces.
It is a special case of \cite[\textsection 1.3]{waldhausen} (cf.~\cite[\textsection IV.8.5.5]{WeibelKBook}) that the $K$-theory space $K(\cE_R)$ can be delooped infinitely many times as follows:

\begin{thm}
\label{delooping}
For $n\geq1$, there is an equivalence of spaces
\[
K(\cE_R)\simeq\Omega^{n} \mathrm{Real}^{(n)}( S^{(n)}_{\bullet,\dots,\bullet}(\cE_R)).
\]
\end{thm}

\subsection{Generalized iterated $S_\bullet$-construction}

In this last section we provide a more general version of the iterated construction that would work for any $\Sigma$-space. This might perhaps allow for an enhancement of Waldhausen's proof of \cref{delooping}, delooping the $S_\bullet$-construction in more general situations.

Let $p^{(n)}=(p,\dots,p)\colon\Sigma\to\Delta\times\dots\times\Delta$ be given by $\sigma\mapsto(p(\sigma),\dots,p(\sigma))$. There is an adjunction
\[\cP^{(n)}=p^{(n),*}\colon\cS^{\Delta^{\op}\times\dots\times\Delta^{\op}}\leftrightarrows\cS^{\Sigma^{\op}}\colon p^{(n)}_*=S_{\bullet,\dots,\bullet}^{(n)}\]
between the category $\cS^{\Delta^{\op}\times\dots\times\Delta^{\op}}$ of $n$-fold simplicial spaces and the category $\cS^{\Sigma^{\op}}$ of $\Sigma$-spaces.

We can understand the functor $\cP^{(n)}$ on the representable functor $\Delta[ k_1,\dots k_n]$:

\begin{prop}
For $n\geq1$ and $ k_1,\dots, k_n\geq0$,
there is an isomorphism of $\Sigma$-spaces:
\[
\cP^{(n)}\Delta[ k_1,\dots k_n]\cong \cP\Delta[ k_1]\times\dots \times\cP\Delta[ k_n]\cong N^{\mathrm{ex}}\Ar[k_1]\times\dots N^{\mathrm{ex}}\Ar[k_n].
\]
\end{prop}

\begin{proof}
Let's do the case $n=2$, the general case working similarly.
For $\alpha$ in $\Sigma$, and $k_1,k_2\geq0$, there are natural isomorphisms of spaces
\[
\begin{array}{llll}
  (\cP^{(2)}\Delta[ k_1, k_2])_\alpha  &\cong&\Map(\Sigma[\alpha],p^{(2),*}\Delta[ k_1, k_2])  \\
    &\cong & \Map(p^{(2)}_!\Sigma[\alpha],\Delta[ k_1, k_2])\\
&\cong&\Map(\Delta[p(\alpha),p(\alpha)],\Delta[ k_1, k_2])\\
&\cong&\Map(\Delta[p(\alpha)],\Delta[ k_1])\times\Map(\Delta[p(\alpha)],\Delta[ k_2])\\
&\cong&\Map(\Sigma[\alpha],\cP\Delta[ k_1])\times\Map(\Sigma[\alpha],\cP\Delta[ k_2])\\
&\cong&(\cP\Delta[ k_1])_\alpha\times(\cP\Delta[ k_2])_\alpha,
\end{array}
\]
as desired.
\end{proof}

We can then describe the functor $S_{\bullet,\dots,\bullet}^{(n)}\colon \cS^{\Sigma^{\op}}\to\cS^{\Delta^{\op}\times\dots\times\Delta^{\op}}$ explicitly:

\begin{const}
 Given a $\Sigma$-space $X$, $n\geq0$ and $ k_1,\dots, k_n\geq0$, we set
 \[S^{(n)}_{ k_1,\dots, k_n}(X)=\Map(\cP\Delta[k_1]\times\dots\times\cP\Delta[k_n],X)\]
 to be the space of maps of $\Sigma$-spaces $\cP\Delta[ k_1]\times\dots \times\cP\Delta[ k_n]\to X$.
\end{const}

So in total we get a multisimplicial space:

\begin{prop}
Given a $\Sigma$-space $X$, the assignment $[ k_1,\dots, k_n]\mapsto S^{(n)}_{ k_1,\dots, k_n}(X) $ defines an $n$-fold simplicial space $S^{(n)}_{\bullet,\dots,\bullet}(X)$.
\end{prop}

This construction recovers the one for exact categories through the exact nerve $N^{\mathrm{ex}}\colon\cE x\cat\to\cS^{\Sigma^{\op}}$ from \cite[Definition~2.2]{BOORS4} as follows:

\begin{rmk}
There is a commutative diagram of categories
\[
\begin{tikzcd}
\cE x\cat\arrow[d,"N^{\mathrm{ex}}" swap]\arrow[r,"S^{(2)}_{\bullet,\bullet}"]&\gpd^{\Delta^{\op}\times\Delta^{\op}}\arrow[d,"N_*"]\\
\cS^{\Sigma^{\op}}\arrow[r,"S^{(2)}_{\bullet,\bullet}" swap]&\cS^{\Delta^{\op}\times\Delta^{\op}}
\end{tikzcd}
\]
Indeed, for $ k_1,k_2\geq0$ and $\cE$ an exact category, there is a natural bijection
\[
\begin{array}{llll}
    S^{(2)}_{ k_1,k_2}(N^{\mathrm{ex}}\cE) &\cong  &\Map(\cP\Delta[ k_1]\times\cP\Delta[ k_2],N^{\mathrm{ex}}\cE)\\
&\cong  &N\Hom^{\mathrm{ex}}(\Ar[ k_1]\times\Ar[ k_2],\cE)\\
&\cong& NS^{(2)}_{ k_1, k_2}(\cE)\cong(N_*S^{(2)})_{ k_1, k_2}(\cE).&\\
\end{array}
\]
\end{rmk}

\begin{rmk}
It is possible to recognize the functor $S^{(n)}_{\bullet,\dots,\bullet}$ as an appropriate iteration of the $S_\bullet$-construction in the context of $\Sigma$-spaces, by replacing the category $\cE x\cat$ with the category $\cS^{\Delta^{\op}}$ in \cref{iteration}.
\end{rmk}

One \emph{could} then hope for a delooping in larger generality, so that $K(X)$ would come from an $\Omega$-spectrum:

\begin{question}
For $n\geq1$, under which conditions on a $\Sigma$-space $X$ does one have equivalence of spaces
\[
K(X) \simeq\Omega^{n} \mathrm{Real}^{(n)} (S_{\bullet,\dots,\bullet}^{(n)}(X))?
\]
\end{question}

This type of statement was proven in \cite[\textsection 1.3]{waldhausen} in the context of Waldhausen categories, and in \cite[\textsection7]{SarazolaShapiro} in the context of ECGW categories, and it might hold in larger generality.

\subsection{Multi-2-Segality properties}

One could explore the $2$-Segality property of the iterated Waldhausen construction. Not much has been done in this direction, and we just include some speculative considerations. The overall expectation is:
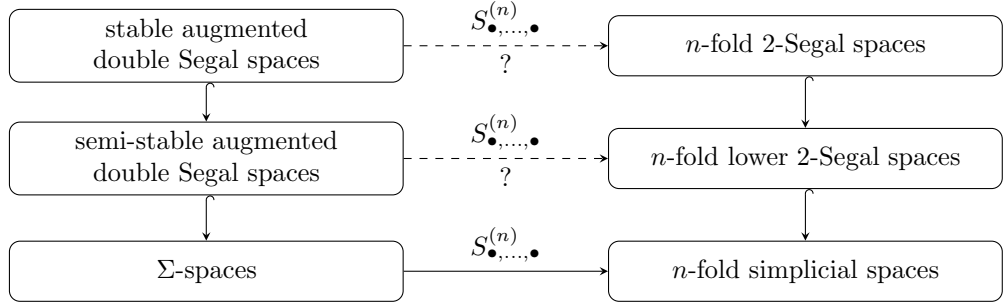
\begin{center}
 \begin{tikzpicture}
 \def\vsp{1.5cm}
 \def\hspo{7cm}
  \def\hspt{4.4cm}
   \draw (-1+\hspo,-2*\vsp) node[category1] (Multi2Segal){$n$-fold $2$-Segal spaces};
   \draw (-1+\hspo, -3*\vsp) node[category1] (Multi2SegPlus){$n$-fold lower $2$-Segal spaces};
    \draw (-1+\hspo, -4*\vsp) node[category1] (MultiSimpSp){$n$-fold simplicial spaces};
   \draw (-1,-2*\vsp) node[category1] (sadss){stable augmented double Segal spaces};
    \draw (-1,-3*\vsp) node[category1] (hsadss){semi-stable augmented double Segal spaces};
      \draw (-1,-4*\vsp) node[category1] (SigmaSp){$\Sigma$-spaces};
   %
\draw[-stealth, dashed] (sadss)--node[above]{$S^{(n)}_{\bullet,\dots,\bullet}$} node[below](label1){$?$} (Multi2Segal);
\draw[-stealth, dashed] (hsadss)--node[above]{$S^{(n)}_{\bullet,\dots,\bullet}$} node[below](label2){$?$} (Multi2SegPlus);
\draw[-stealth] (SigmaSp)--node[above]{$S^{(n)}_{\bullet,\dots,\bullet}$}(MultiSimpSp);
\draw[right hook-stealth] (sadss)--(hsadss);
\draw[right hook-stealth] (hsadss)--(SigmaSp);
\draw[right hook-stealth] (Multi2Segal)--(Multi2SegPlus);
 \draw[right hook-stealth] (Multi2SegPlus)--(MultiSimpSp);
 \end{tikzpicture}
\end{center}

\subsubsection{The multi-$2$-Segal case}

In the pointed stable case we should have (case $n=2$ and $\cE$ being a protoexact category is 
\cite[Corollary~3.17]{Penney}):

\begin{expectation}
If $\cE$ is any of the structures discussed in \cref{stable}, the iterated $S_\bullet$-construction $S^{(n)}_{\bullet,\dots,\bullet}(\cE)$
is a $2$-Segal object in each simplicial variable.
\end{expectation}

For instance, assuming that $X^{\cP\Delta[\ell]}$ is a stable augmented double Segal space whenever $X$ is, \cref{2SegalityThm} would imply that $S^{(2)}_{\bullet,\ell}(X)= S_\bullet(X^{\cP\Delta[\ell]})$ is a $2$-Segal space.

One could then study whether all multi-$2$-Segal spaces arise as an instance of an $S_\bullet$-construction:

\begin{question}
Does the $S_\bullet$-construction define an equivalence between multi-$2$-Segal spaces and stable
augmented double Segal spaces?
If not, is it injective (up to homotopy)? And what is its (homotopy essential) image?
\end{question}

This question is possibly treatable by adjusting the argument from \cite{BOORS3} (or \cite{BOORS1} for the discrete case).

\subsubsection{The multi-lower-$2$-Segal case}

    If $\cW$ is any structure from \cref{NonStable}, the $S_\bullet$-construction is not expected to define a multi-$2$-Segal space, as flaws similar to  those highlighted in \cref{Snot2Segal} would remain. However:

\begin{expectation}
If $\cW$ is any of the structures discussed in \cref{NonStable} the iterated $S_\bullet$-construction $S^{(n)}_{\bullet,\dots,\bullet}(\cW)$
    is a lower $2$-Segal object in each simplicial variable.
\end{expectation}

For instance, assuming that $X^{\cP\Delta[\ell]}$
is a semi-stable augmented double Segal space whenever $X$ is, \cref{Lower2Segal} would imply that $S^{(2)}_{\bullet,\ell}(X)= S_\bullet(X^{\cP\Delta[\ell]})$ is a lower $2$-Segal space.

One could study whether all multi-lower-$2$-Segal spaces arise as $S_\bullet$-construction:

\begin{question}
Does the $S_\bullet$-construction define an equivalence between multi-lower-$2$-Segal spaces and semi-stable augmented
double Segal spaces? If not, is it injective (up to homotopy)? And what is its (homotopy essential) image?
\end{question}

Again, this question is possibly treatable by adjusting the argument from \cite{BOORS3} (or \cite{BOORS1} for the discrete case).

\bibliographystyle{amsalpha}
\bibliography{ref}

\end{document}